\documentclass[11pt,reqno]{article}

\usepackage{times}
\usepackage{amsmath,amsfonts,amstext,amssymb,amsbsy,amsopn,amsthm,eucal}
\usepackage{txfonts}
\usepackage{dsfont}
\usepackage{graphicx}   
\usepackage{hyperref}
\usepackage{color}
\usepackage{verbatim}
\usepackage{enumerate}
\usepackage{tocbibind}

\numberwithin{equation}{section}

\hypersetup{
  pdftitle={Singular set of approximate harmonic maps},
  pdfauthor={Aaron Naber and Daniele Valtorta},
  pdfsubject={},
  pdfkeywords={},
  pdfpagelayout=SinglePage,
  pdfpagemode=UseOutlines,
  colorlinks,
  bookmarksopen,
  linkcolor=[rgb]{0,0,0.7},
  urlcolor=[rgb]{0,0,0.4},
  citecolor=[rgb]{0.4,0.1,0}
}

\newcommand{\Vol}{\text{Vol}}
\newcommand{\diam}{\text{diam}}

\newcommand{\inj}{\text{inj}}
\newcommand{\Sing}{\cS}

\newcommand{\dR}{\mathds{R}}

\newcommand{\cC}{\mathcal{C}}

\newcommand{\cF}{\mathcal{F}}

\newcommand{\cH}{\mathcal{H}}

\newcommand{\cS}{\mathcal{S}}

\newcommand{\norm}[1]{\left\|#1\right\|}
\newcommand{\ps}[2]{\left\langle#1,#2\right\rangle}
\newcommand{\ton}[1]{\left(#1\right)}
\newcommand{\qua}[1]{\left[#1\right]}
\newcommand{\cur}[1]{\left\{#1\right\}}
\newcommand{\abs}[1]{\left|#1\right|}
\newcommand{\wto}{\rightharpoonup ^*}

\newcommand{\B}[2]{B_{#1}\ton{#2}}

\newcommand{\supp}[1]{\operatorname{supp}\ton{#1}}

\newcommand{\N}{\mathds{N}}

\newcommand{\R}{\mathds{R}}

\renewcommand{\paragraph}[1]{\ \newline \ \textbf{#1\ }}

\newcommand{\hol}{H\"older }
\newcommand{\al}{Ahlfors }

\newcommand{\KN}{N} 

\newtheorem{theorem}{Theorem}[section]

\newtheorem{proposition}[theorem]{Proposition}
\newtheorem{lemma}[theorem]{Lemma}
\newtheorem{corollary}[theorem]{Corollary}
\theoremstyle{definition}
\newtheorem{definition}[theorem]{Definition}
\theoremstyle{remark}
\newtheorem{remark}{Remark}[section]
\theoremstyle{remark}
\newtheorem{example}{Example}[section]
\theoremstyle{remark}

\theoremstyle{remark}

\theoremstyle{remark}

\begin{document}

\title{Stratification for the singular set of approximate harmonic maps}
\author{Aaron Naber and Daniele Valtorta}

\date{\today}

\maketitle

\begin{abstract}
The aim of this note is to extend the results in \cite{nava+} to the case of approximate harmonic maps.  More precisely, we will proved that the singular strata $\cS^k(u)$ of an approximate harmonic map are $k$-rectifiable, and we will show effect bounds on the quantitative strata.  In the process we will simplify many of the arguments from \cite{nava+}, and in particular we produce a new main covering lemmas which vastly simplifies the older argument.
\end{abstract}

\tableofcontents

\section{Introduction}

In this note, we will work with maps between two compact Riemannian manifolds $M$ and $N$ such that 
\begin{align}\label{eq_MNprop}
&|\sec_{B_2(p)}|\leq K_M,\,\, \inj(B_2(p))\geq K_M^{-1}\, ,\notag\\
&\partial N = \emptyset ,\ \ |\sec_N|\leq K_N,\,\, \inj(N)\geq K_N^{-1},\,\,\diam(N)\leq K_N\, ,\\
&\dim(M)=m,\, \dim(N)\leq n\, .\notag 
\end{align}
Moreover, we will always assume that $N$ is isometrically embedded into $\R^n$.

A harmonic map $u\in H^1(M,N)$ is a critical point of the Dirichlet energy for fixed boundary values. This map satisfies the Euler-Lagrange equation
\begin{gather}
 \Delta u+A(u)(\nabla u,\nabla u)=0\, ,
\end{gather}
where $A$ is the second fundamental form of $N\subset \R^n$. An important object in the study of harmonic maps is the normalized energy
\begin{gather}
 \theta(x,r)=r^{2-m}\int_{\B r x} \abs{\nabla u}^2\, .
\end{gather}
This quantity turns out to be (almost) monotone for \textit{stationary} harmonic maps. Combined with an $\epsilon$-regularity theorem, this immediately yields the partial regularity of such maps.

Loosely speaking, approximate harmonic maps are solutions to
\begin{gather}
 \Delta u+A(u)(\nabla u,\nabla u)=f\, ,
\end{gather}
where $f$ in an $L^2$ function. Such a map is called \textit{stationary} if it also satisfies 
\begin{gather}
  \nabla^i\ton{\abs{\nabla u}^2 g_{ij} -2 \ps{\nabla_i u}{\nabla_j u}}+2\ps{\nabla_j u}{f}=0\, .
\end{gather}

As for harmonic maps, one defines the singular set of the map $u$ as the set of points in the domain which are not continuous:
\begin{gather}
 \Sing(u)=\cur{x\in M \ \ s.t. \ \ \exists r>0 \ \ s.t. \ \ u|_{\B r x} \ \ \text{ is continuous }}^C\, .
\end{gather}
By elliptic regularity, it is easy to see that around a regular point almost harmonic maps are $C^{0,\alpha}$ continuous, where $\alpha=\alpha(m,K_M,\KN,p)>0$, but in general higher regularity depends on the regularity of $f$.

As for standard harmonic maps, one can define \textit{stationary} almost harmonic maps and prove that the normalized energy $\theta(x,\cdot)$ is almost monotone, in the sense that has bounded variation at every point.
Using only this property, and an $\epsilon$-regularity theorem, one shows that $\Sing(u)$ has zero $m-2$-Hausdorff measure and that each of these maps has a (possibly non-unique) tangent map at every point.

By looking at the symmetries of these tangents, see Section \ref{ss:stratification} for precision, one can define a stratification $\cS^k(u)$ for $\cS(u)$ by setting
\begin{gather}
 \cS^k(u)=\cur{x\in \B 1 0 \ \ s.t. \ \ \text{no tangent map at} \ x \ \text{is k-symmetric}}\, .
\end{gather}
The aim of this note is to prove that the results obtained in \cite{nava+} for stationary harmonic maps continue to hold in this setting, and in particular the strata $\cS^k(u)$ are all $k$-rectifiable.  Our full collection of main theorems is listed in Section \ref{s:main_theorems},  after some preliminaries are introduced, however we begin by stating the main structure result for the stratification itself:

\begin{theorem}
Let $u:B_3(p)\subseteq M\to N$ be a stationary approximately harmonic map, where $M$ and $N$ satisfy \eqref{eq_MNprop}, and $f\in L^p$ for $p>\frac{m}{2}$.  Then for each $k$ we have that $S^k(u)$ is $k$-rectifiable.
\end{theorem}

In fact, in Section \ref{s:main_theorems} we will weaken slightly the assumptions on $f$.  Moreover, we will also prove in Section \ref{s:main_theorems} uniform volume bounds for the \textit{quantitative} strata $\cS^k_{\epsilon,r}$, which will be introduced in the next section.\\

There are three key elements in the proof of these estimates: an estimate on the $\beta_2$-Jones' numbers for the singular strata derived using the monotonicity formula and the $L^2$-best subspace approximation theorem, a Reifenberg-type theorem which allows us to turn these estimates into volume bounds and rectifiability for the singular strata, and an inductive covering argument which guarantees the applicability of this Reifenberg theorem on the strata. While all the results and their proofs are essentially the same as in \cite{nava+}, albeit done in slightly more generality, it is worth noticing that the proof inductive covering argument has been simplified here.

\subsection{Domains with curvature}
For the sake of simplicity, throughout the rest of this note we will assume that $M=\Omega\subseteq \R^m$. Assuming that the domain $M$ is flat simplifies the technicalities involved in the computations, but involves no fundamental changes in comparison to the general case - primarily, one obtains errors from the almost aspect of the monotonicity formula which simply need to be accounted for in a standard manner.  Thus, all the results described in this note carry over to the general case, up to minor technical differences. In particular, all the constants involved in the bounds would also depend on the manifold $M$, and actually just on a lower bound on its injectivity radius and an upper bound on the sectional curvature. Moreover, the normalized energy $\theta$, and its adapted version for the almost harmonic case $\hat \theta$, would need to be slightly changed as described in Remark \ref{rem_curv}.

\subsection{Quantitative stratification for general maps}\label{ss:stratification}
In this section, we introduce in detail the quantitative stratification of the singular set $\cS(u)$. As mentioned in the introduction, the idea behind the quantitative stratification is to group all points in the domain of the map $u$ according to the number of approximate symmetries of $u$ at some scale, as opposed to the standard stratification which looks at the exact symmetries of the set of tangent maps.

A first version of the quantitative stratification can be found in Almgren's big regularity paper (see \cite[section 2.25]{almgren_big}), and it was reintroduced in \cite{ChNa1,ChNa2} to provide the first effective regularity results. In particular, in \cite{ChNa1,ChNa2} the authors use the quantitative stratification to prove estimates on noncollapsed manifolds with Ricci curvature bounds, and on the singular set of stationary and minimizing harmonic maps. As a corollary of the estimates, in \cite{ChNa2} the authors obtain uniform $W^{1,p}$ bounds for minimizing harmonic maps for all $p<3$.  This technique has since been used in \cite{ChNaHa1}, \cite{ChNaHa2}, \cite{ChNaVa}, \cite{FoMaSpa}, \cite{brelamm} to prove similar results in the areas of mean curvature flow, critical sets of elliptic equations, harmonic map flow, and biharmonic maps.

A significant improvement on these techniques has been obtained in \cite{nava+}, where the authors use sharp estimates on the Jones' $\beta_2$ numbers (called \textit{distortion} in \cite{nava+}) in terms of the normalized energy to prove sharp volume bounds and rectifiability for the singular strata of harmonic maps.  As mentioned in the introduction, the aim of this note is to extend these results to the case of approximately harmonic maps, but perhaps more importantly to give somewhat simplified proofs of the main arguments of \cite{nava+}.

\vspace{1mm}

In order to define precisely the quantitative stratification, we need to introduce the concept of $k$-symmetric maps.

\begin{definition}
Given a map $h\in H^1(\R^m,N)$, we say that
\begin{enumerate}
 \item $h$ is homogeneous wrt the point $p$ if $h(p+\lambda v) = h(p+v)$ for all $\lambda>0$ and $v\in \R^m$. Equivalently, $\partial_{r_p} h =0$, where $r_p$ is the radial direction wrt $p$,
 \item $h$ is $k$-symmetric if it is homogeneous wrt the origin and it has an invariant $k$-subspace, i.e., if there exists a linear subspace $V\subseteq \R^m$ of dimension $k$ such that
 \begin{gather}
  h(x+v)=h(x) \ \ \forall x\in \R^m \ \ \forall v\in V\, .
 \end{gather}
\end{enumerate}
\end{definition}
As a notation, we say that $h$ is $0$-symmetric iff it is homogeneous wrt the origin. In the definition, we insist that a $k$-symmetric map be \textit{both homogeneous and $k$-invariant}.

\begin{example}
It is very easy to produce examples of these maps by taking maps defined on $S^{m-1}$ and extending them by homogeneity on the whole $\R^m$. Note that necessarily if a map is homogeneous, then it is continuous only if it is constant.

The first nonconstant explicit example is the map $h:\R^m \to S^{m-1}$ given by
\begin{gather}
 h(x)= \frac{x}{\abs x}\, ,
\end{gather}
which is a $H^1$ map for $m\geq 3$ (and actually a \textit{minimizing} harmonic map in these cases, see \cite{lin_min,corgul}). We can easily build from this example a $k$-symmetric map by defining $g:\R^m\times \R^k \to S^{m-1}$ as $g(x,y)=h(x)$.
\end{example}

The quantitative stratification is based on how close a map is to some $k$-symmetric map at different scales. 
\begin{definition}
 Given a map $g\in H^1(\Omega, N)$, we say that $\B r x \subset \Omega$ is $(k,\epsilon)$-symmetric for $g$ if there exists some $k$-symmetric function $h$ such that
 \begin{gather}
  \fint_{\B r x} \abs{g(y)-h(y-x)}^2\equiv \frac{1}{\omega_m r^m}\int_{\B r x} \abs{g(y)-h(y-x)}^2 \leq \epsilon\, . 
 \end{gather}
\end{definition}
An equivalent definition may be given in terms of the blow-up map $T^g_{x,r}$, defined by
\begin{gather}
 T^g_{x,r} (y)= g(x+ry)\, .
\end{gather}
Indeed, since $h$ is homogeneous by assumption, we have
\begin{gather}
 \fint_{\B r x} \abs{g(y)-h(y-x)}^2=\fint_{\B 1 0 } \abs{T^g_{x,r}(y)-h(y)}^2\, .
\end{gather}

Given this definition of approximate symmetry, we can define the quantitative stratification by classifying the points in the domain according to how close they look at different scales to something which is $k$-symmetric.

\begin{definition}\label{d:stratification}
Given $u\in H^1(\Omega,N)$, $r,\epsilon>0$ and $k\in \cur{0,1,\cdots,m}$, we define
\begin{gather}
 \cS^k_{\epsilon,r}(u)\equiv \cur{x\in \Omega \ \ s.t. \ \ \text{ for no }r\leq s<1, \ \text{$\B{s}{x}$ is } \ (k+1,\epsilon)\text{-symmetric wrt} \ u}\, .
\end{gather}
\end{definition}
Note that these sets have some immediate inclusion properties coming from their definition. In particular, if $k'\leq k$, $\epsilon'\geq \epsilon$ and $r'\leq r$, we have
\begin{gather}
 \cS^{k'}_{\epsilon',r'}(u)\subseteq \cS^k_{\epsilon,r}(u)\, .
\end{gather}

Given this, one can construct the sets
\begin{gather}
 \cS^k_\epsilon(u) = \bigcap_{r>0} \cS^k_{\epsilon,r}(u)\, , \quad \cS^k(u)=\bigcup_{\epsilon>0} \cS^k_{\epsilon}(u)\, .
\end{gather}

\begin{remark}
Note that, for an approximate harmonic map, the set $\cS^k(u)$ has another characterization:
\begin{gather}
 \cS^k(u)= \cur{x\in \Omega \ \ s.t. \ \ \text{no tangent maps at } x \text{ is } k+1-\text{symmetric}}\, .
\end{gather}
For a precise statement, see Lemma \ref{lemma_Sk}.
\end{remark}

\begin{example}
 It is interesting to study these sets with an example. In particular, consider the map $g:\R^3\to S^{2}$ given by $g(x)=x/\abs x$. Since the map $g$ is $0$-homogeneous but not $1$-symmetric, it is clear that $0\in \cS^0_{\epsilon,r}$ for all $r$ and $\epsilon< \epsilon_0$ sufficiently small. In particular, $\epsilon_0$ can be taken to be the $L^2(\B 1 0)$ distance from $g$ to all $1$-symmetric maps $h$, which is easily seen to be positive.
 
 Now given any $x\in \R^n\setminus \{0\}$, since $g$ is continuous at $x$, for all $\epsilon>0$ there exists a radius $r(\epsilon)$ such that 
 \begin{gather}
  \fint_{\B {r\abs x} x} \abs{g(y)-g(x)}^2\leq \epsilon\, .
 \end{gather}
Thus every point $x\neq 0$ will eventually be almost $3$-symmetric, if we consider a small enough radius.  However, for $r$ sufficiently big, $g|_{\B r x}$ and $g|_{\B r 0}$ are close in the $L^2$ norm. Thus, for $r\geq s(\abs x,\epsilon)$ sufficiently large we have that $x$ will belong to $\cS^0_{\epsilon,r}$.  Summing up, we obtain that
\begin{gather}
 \cS^0_{\epsilon,r}(g)= \overline{\B {s(r,\epsilon)}{0}}\,  \quad \text{ where } \ \ s(r,\epsilon) \ \ \text{ is increasing in $r$ and decreasing in $\epsilon$}\, .
\end{gather}
Moreover, for $\epsilon<\epsilon_0$, we have $s(0,\epsilon)=\lim_{r\to 0} s(0,\epsilon)=0$.  Up to different constants, the sets $\cS^1_{\epsilon,r}(g)$ and $\cS^2_{\epsilon,r}(g)$ also behave in a similar way, while evidently $\cS^3_{\epsilon,r}(g)=\R^3$ (but this last statement has clearly no meaning).
\end{example}

\section{Approximate harmonic maps: definition and monotonicity}
Before moving to approximate harmonic maps, we briefly recall the most important aspects of the regularity theory for harmonic maps.

\subsection{Stationary harmonic maps: regularity}

Loosely speaking, harmonic maps are critical points of the Dirichlet energy, in particular with respect to compact variations in the target space and in the domain space, in the case of stationary harmonic maps. The definition of these object is standard in literature, here we briefly recall it.

Given a compact $C^2$ Riemannian manifold $N$, we define the Sobolev space $H^1(M,N)$ by isometrically embedding $N$ into a Euclidean space $\R^n$, and considering
\begin{gather}
 H^1(\Omega,N)=\cur{ \ \ f\in H^1(\Omega,\R^n) \ \ s.t. \ \ \text{for a.e.} \ x\in M\, , \ f(x)\in N }\, .
\end{gather}

Let $\Omega\subset \R^m$ be a domain. A stationary harmonic map $u\in H^1(\Omega,N)$ is a critical point for this Dirichlet energy. In particular, let $v:M\to \R^n$ and $w:\Omega\to \R^m$ be smooth vector fields, both with compact support in $\Omega$. Then if $u$ is a stationary harmonic map we have
\begin{gather}
 \left.\frac{d}{dt}\right\vert_{t_0} \int_{\Omega} \abs{\nabla \ton{\Pi_N(u+tv)}}^2=0 \quad \text{and}\quad \left.\frac{d}{dt}\right\vert_{t_0} \int_{\Omega} \abs{\nabla\ton{u(x+tw)}}^2=0 \, ,
\end{gather}
where $\Pi_N$ is the nearest point projection from $\R^n$ onto $N$, a map which is well-defined and $C^1$ on a small neighborhood of $N$. 

By standard computations (see for example \cite{xin,moser_appr_harm}), the first condition gives the Euler-Lagrange equation
\begin{gather}\label{eq_Au}
 \Delta u +A(u)(\nabla u,\nabla u)=0\, .
\end{gather}
This is an equation satisfied by $u:\Omega \to \R^n$ in the weak $H^1$ sense, and $A(u)$ is the second fundamental form of the embedding $N \hookrightarrow \R^n$ evaluated at $u(x)$. Maps satisfying only this equation are called \textit{weakly harmonic}.

The second condition gives rise to a separate Euler-Lagrange equation. In particular, we see that for all $i=1,\cdots,m$ we have
\begin{gather}\label{eq_stat_short}
 \operatorname{div_j}\ton{\abs{\nabla u}^2 \delta_{ij} -2 \ps{\nabla_{i} u}{\nabla_j u}}=0\, .
\end{gather}
Also this equation is to be interpreted in the weak sense, in particular we have that for all smooth vector fields with compact support $\xi:\Omega \to \R^m$
\begin{gather}\label{eq_stat}
 \int_\Omega \ton{\abs{\nabla u}^2 \delta_{ij} -2\partial_i u \partial_j u}\partial_i \xi_j=0\, .
\end{gather}
Note that a map $u\in C^2(\Omega,\R^n)$ satisfying \eqref{eq_Au} will automatically satisfy also \eqref{eq_stat}. However, in general this is not the case if $u\in H^1$ only. Maps satisfying both \eqref{eq_Au} and \eqref{eq_stat} are called \textit{stationary} harmonic maps.

An important consequence of \eqref{eq_stat} is the monotonicity of the normalized energy $\theta$ defined by
\begin{gather}
 \theta(x,r)= r^{2-m} \int_{\B r x} \abs{\nabla u}^2\, .
\end{gather}
\begin{lemma}
 For a stationary harmonic map $u\in H^1(\Omega,\R^n)$, and for almost all $x,r$ such that $\B r x \subset \Omega$, we have the monotonicity identity
 \begin{gather}\label{eq_theta'}
 \theta'(x,r)= 2r^{-m} \int_{\partial \B r x} \abs{(y-x)\cdot \nabla u}^2 dS(y) \geq 0\, .
 \end{gather}
 This in particular implies that $\theta(x,r)$ is a monotone function in $r$, and it also gives quantitative estimates for the integral of the radial part of the energy. Indeed, as a corollary we obtain that for all $0\leq s\leq r$:
 \begin{gather}
  2\int_{\B r x \setminus \B s x} \frac{\abs{(y-x)\cdot \nabla u}^2}{\abs{y-x}^{m}} = \theta(x,r)-\theta(x,s)\, .
 \end{gather}

\end{lemma}
\begin{proof}
 This identity is obtained by plugging in the vector field $\xi_i= \chi_{\B r x }\cdot x_i$ in \eqref{eq_stat} (or better, a sequence of smooth approximations of this field).
\end{proof}

The last ingredient needed for a very basic regularity theory for stationary harmonic maps is the $\epsilon$-regularity theorem by Bethuel (see \cite{beth}).
\begin{theorem}[$\epsilon$-regularity theorem]
 Let $u\in H^1$ be a stationary harmonic map. There exists an $\epsilon=\epsilon (m,\KN)>0$ such that $\theta(x,r)\leq \epsilon$ implies that $u$ is smooth on $\B {r/2}{x}$ with 
 \begin{gather}
  \abs{\nabla^k u(y)}\leq C(m,k,\KN) r^{-k} \theta(x,r)\, .
 \end{gather}
for all $k=1,\cdots,\infty$ and $y\in \B {r/2}{x}$.
\end{theorem}

With these two theorems and a simple covering argument (see \cite{beth}), it is easy to see that
\begin{gather}
 \cH^{n-2}(\cS(u))=0\, .
\end{gather}
Note that in general this is not the case for weakly harmonic maps.

\subsection{\texorpdfstring{Monotonicity and $\epsilon$-regularity}{Monotonicity and epsilon-regularity}}

Following the natural approach of \cite{moser_appr_harm}, here we introduce approximately harmonic maps. In this note, approximately harmonic maps will always mean \textit{stationary} approximately harmonic maps.

Basically, we say that a map is approximately harmonic if it satisfies equations \eqref{eq_Au} and \eqref{eq_stat_short} up to a controlled error. 
\begin{definition}
 A map $u\in H^1(\Omega,N)$ is said to be approximately harmonic if the following are satisfied in the sense of distributions
 \begin{gather}
  \Delta(u)+A(u)(\nabla u,\nabla u)=f\, ,\label{eq_app1}\tag{W}\\
  \operatorname{div_j}\ton{\abs{\nabla u}^2 \delta_{ij} -2 \ps{\nabla_i u}{\nabla_j u}}+2\ps{\nabla_i u}{f}=0\, ,\label{eq_app_stat}\tag{S}
 \end{gather}
 for some $f\in L^2(\Omega)$. In order to obtain almost monotonicity for the normalized energy of the solution $u$, we assume that there exists $\cF,\gamma>0$ such that for all $\B r x \subset \Omega$: 
 \begin{gather}
  r^{4-m} \int_{\B r x} \abs f ^2 \leq \cF r^\gamma\, ,\label{eq_f}\tag{f}
 \end{gather}
where $\gamma>0$, and $\cF$ and $\gamma$ are independent of $x$.  One can interpret $(f)$ as saying some weighted maximal function of $f$ is uniformly bounded.  Although this assumption might sound extremely technical or unnatural, it is easy to see that any $f\in L^p(\Omega)$ with $p>m/2$ satisfies this condition. Indeed, by \hol inequality we have
\begin{gather}
 r^{4-m} \int_{\B r x} \abs f ^2 \leq  r^{4-m} \ton{\int_{\B r x} \abs f ^p}^{\frac 2 p} \ton{\omega_m r^m}^{\frac{p-2}{p}} \leq c(m,p) \norm{f}_{L^p(\Omega)}^2 r^{\frac 2 p (2p-m)}\, .
\end{gather}

\end{definition}

\begin{remark}
Note that the condition \eqref{eq_app_stat} is quite natural, since it would be satisfied automatically by a smooth map $u$ solving \eqref{eq_app1}. However, in general an $H^1$ solution to \eqref{eq_app1} will \textit{not} satisfy also \eqref{eq_app_stat}.
\end{remark}

The reason why we insist on \eqref{eq_app_stat} is that with this relation we can prove an almost monotonicity formula for approximately harmonic maps, which is essential for the estimates we need. The following lemma is taken from \cite[lemma 4.1]{moser_appr_harm}, however the quantities analyzed are slightly different. For the reader's convenience, we sketch a proof here.
\begin{lemma}\label{lemma_monotone}
For $m\geq 3$, let $u\in H^1(\B 3 0,N)$ be an approximately harmonic map satisfying \eqref{eq_app1} and \eqref{eq_app_stat} with \eqref{eq_f}. Suppose also that
\begin{gather}
 \theta(0,3)=3^{2-m}\int_{\B 3 0}\abs {\nabla u}^2 \leq \Lambda\, .
\end{gather}
Define the function
\begin{gather}\label{eq_deph_hat}
 \hat \theta_u(x,r)\equiv \hat \theta(x,r)= \theta(x,r) -\frac{2}{m-2}r^{2-m}\int_{\B r x} \ps{(y-x)\cdot \nabla u|_y}{f(y)} d\Vol(y)+ \frac{1}{(m-2)^2} \int_{\B r x} \frac{\abs{f}^2}{\abs {y-x}^{m-4}} d\Vol(y)\, .
\end{gather}
Then for all $\B s x \subseteq \B r x $ with $r\leq 1$ and $x\in \B 1 0$ we have $\hat \theta(x,r)\geq 0$ and 
\begin{gather}\label{eq_hat'}
 0\leq \int_{\B r x \setminus \B s x } \frac{\abs{(y-x)\cdot \nabla u}^2}{\abs{y-x}^m}\leq \hat \theta(x,r)-\hat \theta(x,s)\, .
\end{gather}
Moreover, for all $\operatorname{a}>0$ we have the bounds
\begin{gather}\label{eq_hatbounds}
(1-\operatorname{a})\hat \theta(x,r) - \frac{c(m,\gamma)\cF}{\operatorname{a}}r^\gamma\leq \theta(x,r)\leq (1+\operatorname{a})\hat \theta(x,r) + \frac{c(m,\gamma)\cF}{\operatorname{a}}r^\gamma\, ,
\end{gather}
and the uniform bounds
\begin{gather}\label{eq_theta_unif}
 \theta(x,r)+\hat \theta(x,r)\leq c(m)\Lambda + c(m,\gamma) \cF r^\gamma\, .
\end{gather}

\end{lemma}
\begin{proof}
 For convenience, we assume $x=0$. First of all, we prove that, under condition \eqref{eq_f}, the function $\hat \theta$ is finite. For $m=3,4$, this is obvious since $f\in L^2$, and \eqref{eq_f} is not needed. For $m\geq 5$, by integration by parts we have
 \begin{gather}
  \int_0^r s^{4-m}\int_{\partial \B s 0} \abs{f}^2 d\theta ds = \qua{s^{4-m}\int_{\B s 0} \abs{f}^2 }^r_0 +\int_{0}^r s^{3-m} dt \int_{\B t 0} \abs f^2\leq\\
 \leq \cF r^\gamma +(m-4)\cF \int_{s}^r t^{3-m} t^{m-4+\gamma} dt \leq \cF\ton{1+\frac{m-4}{\gamma}}r^\gamma\, .\label{eq_mono1}
 \end{gather}

 As for the monotonicity of $\hat \theta$, the proof is a simple application of the stationary equation \eqref{eq_app_stat}. Indeed, let $\phi$ be any Lipschitz radial cutoff function with $\phi(0)=1$ and $\phi(r)=0$, and consider the vector field $\xi(y)=\phi(\abs {y}) \vec y$. By testing \eqref{eq_app_stat} with $\xi$, we get
  \begin{gather}
   \int_{\Omega} \phi \ton{(m-2) \abs{\nabla u}^2 -2\ps{y\cdot \nabla u}{f} } = \int_{\Omega} \ton{-\phi'} \abs y \ton{\abs{\nabla u}^2 -2  \abs{\hat r\cdot \nabla u}^2 }
  \end{gather}
 where $\hat r = \abs y ^{-1} y $ is the unit norm radial vector. By letting $\phi$ converge to $\chi_{\B r 0}$, we prove that for almost all $r$
 \begin{gather}\label{eq_stat1}
  \int_{\B r 0} \ton{(m-2) \abs{\nabla u}^2 -2\ps{y\cdot \nabla u}{f} } = r \int_{\partial \B r 0} \ton{\abs{\nabla u}^2 -2  \abs{\hat r\cdot \nabla u}^2 } \, .
 \end{gather}
The derivative of $\hat \theta$ is, at least a.e. in $r$,
\begin{gather}
 \hat \theta(0,r)' = \frac{2-m}{r} \theta(0,r) + r^{2-m}\int_{\partial \B r 0} \abs{\nabla u}^2+2r^{2-m}\int_{\B r 0} \ps{y\cdot \nabla u}{f} +\\
 -\frac{2}{m-2}r^{2-m}\int_{\partial \B r 0} \ps{y\cdot \nabla u}{f} +\frac{1}{(m-2)^2} r^{4-m}\int_{\partial \B r 0}\abs f^2\, .
\end{gather}
By plugging in \eqref{eq_stat1}, we obtain for a.e. $r$
\begin{gather}
 \hat \theta(0,r)' = r^{-m}\ton{2\int_{\partial \B r 0}\abs{y\cdot \nabla u}^2 -\frac{2}{m-2} r^2 \int_{\partial \B r 0} \ps{y\cdot \nabla u}{f} +\frac{1}{(m-2)^2}r^4 \int_{\partial \B r 0}\abs f^2}\geq\\
 \geq r^{-m}\int_{\partial \B r 0}\abs{y\cdot \nabla u}^2\, .
\end{gather}
where the last estimate is a simple application of Young's inequality. By integrating this relation over $[s,r]$, we obtain \eqref{eq_hat'}.

As for \eqref{eq_hatbounds}, this follows immediately from Young's inequality applied to the rhs of \eqref{eq_deph_hat} and \eqref{eq_mono1}. Indeed:
\begin{gather}\label{eq_mono2}
 \abs{\frac{2}{m-2}r^{2-m}\int_{\B r x} \ps{(y-x)\cdot \nabla u|_y}{f(y)} }\leq \operatorname{a} \ \theta(x,r) + \frac{r^{4-m}}{(m-2)^2\operatorname{a}} \int_{\B r x} \abs {f}^2\, .
\end{gather}

We conclude by noticing that since $\hat\theta$ is monotone in $r$, and since by \eqref{eq_mono1} and \eqref{eq_mono2}:
\begin{gather}
 \hat\theta(x,1){\leq} 2\theta(x,1) + \cF+ c(m,\gamma) \cF \leq c(m) \Lambda + c(m,\gamma) \cF\, ,
\end{gather}
then the bounds on $\hat \theta$ are obvious. The uniform bounds on $\theta$ are then a consequence of the previous estimate \eqref{eq_hatbounds}.
\end{proof}

\begin{remark}\label{rem_curv}
 In case the domain space is a Riemannian manifold $M$, the definition of $\hat \theta$ would need to be changed a little. We refer the reader to \cite[section 2.2]{xin} for more details on this.  We just mention that the changes arise from the fact that the Hessian of the distance function $r$ is a little different than in the Euclidean case. This is related to the derivatives of the radial vector field $\xi$ used in the proof of the monotonicity formula.
 
\end{remark}

\vspace{3mm}

For approximate harmonic maps, it is also possible to prove an $\epsilon$-regularity theorem as for stationary harmonic maps. The underlying techniques are basically the same, up to minor technical details. Here we quote the $\epsilon$-regularity theorem in \cite[proposition 4.1]{moser_appr_harm}.
\begin{theorem}\label{th_app_eps_reg}
 Let $u$ solve \eqref{eq_app1} and \eqref{eq_app_stat} with \eqref{eq_f}. Then there exists $\epsilon_0,\alpha>0$ depending only on $m,\KN,\gamma$ such that
 \begin{gather}
  \theta(x,r)\leq \epsilon_0 \quad \Longrightarrow \quad u\in C^{0,\alpha} (\B {r/2}{x}) \ \ \text{with} \ \ \norm{u}_{C^{0,\alpha}}\leq C(m,\KN,\cF,\gamma)\, .
 \end{gather}
\end{theorem}
\begin{proof}
 The proof is based on a polynomial decay (in $r$) for the normalized energy $\theta(x,r)$ which is very similar to the proof of the $\epsilon$-regularity theorem for stationary harmonic maps. For the details, we refer the reader to \cite[proposition 4.1]{moser_appr_harm}.
\end{proof}

As a corollary, we obtain a similar statement for $\hat \theta$.
\begin{corollary}\label{cor_app_eps_reg}
 Let $u$ be as above. Then there exists $\epsilon_0,r_0,\alpha>0$ depending only on $m,\KN,\gamma$ such that
 \begin{gather}
  \hat \theta(x,r)\leq \epsilon_0 \ \ \text{ and } \ \ r\leq r_0 \quad \Longrightarrow \quad u\in C^{0,\alpha} (\B {r/2}{x}) \ \ \text{with} \ \ \norm{u}_{C^{0,\alpha}}\leq C(m,\KN,\cF,\gamma)\, .
 \end{gather}
\end{corollary}
\begin{proof}
 This corollary follows immediately from the estimates in \eqref{eq_hatbounds} and the previous proposition.
\end{proof}

As with stationary harmonic maps, this $\epsilon$-regularity theorem and a simple covering argument imply that
\begin{gather}
 \cH^{n-2}(\cS(u))=0\, .
\end{gather}

\paragraph{Invariance by scale}
As it is well-known, for stationary harmonic map the normalized energy $\theta$ is scale-invariant, in the sense that if we consider the map
\begin{gather}
 T^u_{x,r}(y)\equiv u(x+ry)\, ,
\end{gather}
then $\theta_T(0,1)=\theta_u(x,r)$, without any scaling factors. This is an essential property of the normalized energy.

In the case of approximate harmonic maps, the quantity $\hat \theta$ satisfies similar properties. However, some scaling factors are inevitably present on the zero order term $f$. Indeed, the map $T\equiv T^u_{x,r}$ in this case will be an approximate harmonic map satisfying
 \begin{gather}
  \Delta(T)+A(T)(\nabla T,\nabla T)=\tilde f\, ,\label{eq_T1}\\
  \operatorname{div}\ton{\abs{\nabla T}^2 e_i -2 \ps{\nabla_{e_i} T}{\nabla T}}+2\ps{\nabla_{e_i} T}{\tilde f}=0\, ,\label{eq_T2}\\
  \tilde f(y) \equiv r^2 f(x+ry)\, . \label{eq_Tf}
 \end{gather}
Thus, we obtain that $\hat \theta_u(x,r)=\hat \theta_T(0,1)$ if we replace $f$ with $\tilde f$ in the definition of $\hat \theta_T$. Note also that if \eqref{eq_f} is satisfied, then for all $\B r x \subset \B 2 0$ we have
\begin{gather}
 \int_{\B 1 0} \abs{\tilde f}^2 = r^4\int_{\B 1 0} \abs{f(x+ry)}^2=r^{4-m}\int_{\B r x} \abs{f(y)}^2 \leq \cF r^\gamma \, . 
\end{gather}

\subsection{Weak convergence}

In this section, we recall a standard result about the convergence of almost harmonic maps. In particular, we want to show that given a sequence of approximate harmonic maps with bounded energy and such that $f_i\to 0$, their weak sub-limit is a weakly harmonic map. This result is an easy adaptation of standard estimates in literature, see for example \cite[theorem 4]{torowang} or \cite[corollary 2.3]{schoen_analytical_aspects}. 

\begin{proposition}\label{prop_weaklim}
 Let $u_i$ solve \eqref{eq_app1} and \eqref{eq_app_stat}, where $f_i\in L^2(\B 3 0)$ satisfies \eqref{eq_f} with $\cF$ and $\gamma$ independent of $i$. Assume that $\int_{\B 3 0}\abs{\nabla u_i}^2\leq C$ and that $f_i\wto 0$ in weak $L^2$. Then there exists a subsequence (which will still be denoted by $u_i$) such that
\begin{enumerate}
 \item $u_i$ converges in the weak $H^1$ sense to some $u$
 \item there exists a close set $\Sigma$ with finite $n-2$ packing content such that the sequence $u_i$ converges strongly in $H^1_{loc}(\B 1 0 \setminus \Sigma)\cap C^{0,\alpha/2}_{loc}(\B 1 0 \setminus \Sigma)$ to $u$, which is a smooth map on $\B 1 0\setminus \Sigma$. 
 \item $u$ is weakly-harmonic (but not necessarily stationary harmonic)
 \item\label{it_uc} $u$ enjoys the unique continuation property, in the sense that if there exists another weakly harmonic map $v$ such that $u=v$ a.e. on an open set, and $v$ is smooth away from a set of $\Sigma'$ with finite $n-2$ packing content, then $u=v$
\end{enumerate}
\end{proposition}

\begin{remark}
 Note that even if $u_i$ are stationary harmonic (i.e. if $f_i=0$ for all $i$), their limit might not be stationary. An interesting example of this is given in \cite{dinglili}.
\end{remark}

\begin{remark}
 As a corollary of this theorem, we obtain the following. Consider the measure $\abs{\nabla u_i}^2 d\Vol$ on $\B 1 0$. Then in the weak sense of measures we have
 \begin{gather}
  \abs{\nabla u_i}^2 d\Vol\wto \abs{\nabla u}^2 d\Vol + \nu\, ,
 \end{gather}
where $\nu$ is a nonnegative measure by Fatou's lemma, which is supported on the set $\Sigma$ described in the previous theorem. In particular, the support of $\nu$ has finite $n-2$ packing content. For stationary harmonic maps, this measure is called \textit{defect measure}, and it has been extensively studied in \cite{lin_stat}.
\end{remark}

Before stating this result, we need a technical lemma about weak $H^1$ convergence in the $\epsilon$-regularity region.
\begin{lemma}\label{lemma_conv}
 Let $u_i$ solve \eqref{eq_app1} and \eqref{eq_app_stat}, where $f_i\in L^2(\B 3 0)$ satisfies \eqref{eq_f} with $\cF$ and $\gamma$ independent of $i$. Assume that
\begin{enumerate}
 \item $\theta_{u_i}(0,2)\leq \epsilon_0$, where $\epsilon_0$ is the parameter in theorem \ref{th_app_eps_reg},
 \item $u_i \wto u$ in the weak $H^1(\B 3 0)$ sense,
 \item $f_i\wto f$ in the weak $L^2$ sense.
\end{enumerate}
Then $u_i$ converges to $u$ in the strong $H^1(\B {1}{0})$ sense and on $\B {1}{0}$ the map $u$ solves
\begin{gather}
 \Delta(u)=A(u)(\nabla u,\nabla u)+f\, 
\end{gather}
in the sense of distributions.
\end{lemma}

\begin{proof}
 The proof relies on standard techniques, but for the sake of completeness we outline it here. Note that by the $\epsilon$-regularity theorem in theorem \ref{th_app_eps_reg}, we have that $\norm{u_i}_{C^{0,\alpha}(\B {1}{0})}\leq C$, with a uniform bound independent of $i$. Since $N$ is a compact manifold, we also have that $\norm{u_i}_{L^\infty(\B 3 0)}$ is uniformly bounded. Thus $u_i$ converges to $u$ in the strong $C^{0,\alpha/2}(\B {1}{0})$ sense.

 First of all, we prove that $\nabla u_i$ converges to $\nabla u$ in the strong $L^2(\B {1}{0})$ sense. For this purpose, it is sufficient to show that for all $\phi\in C^\infty_C(\B 1 0)$,
 \begin{gather}
  \int_{\B 1 0} \abs{\nabla (u_i-u)}^2\phi \to 0\, .
 \end{gather}
We can split this integral as
\begin{gather}
 \int \abs{\nabla (u_i-u)}^2\phi = \int \ps{\nabla (u_i-u)}{\nabla (u_i-u)}\phi=\int \ps{\nabla (u_i-u)}{\nabla u_i}\phi-\int \ps{\nabla (u_i-u)}{\nabla u}\phi\, .
\end{gather}
By weak convergence, the second integral tends to $0$ with $i$. As for the first, we have
\begin{gather}
 -\int \ps{\nabla (u_i-u)}{\nabla u_i}\phi = \int \ps{u_i-u}{\Delta u_i}\phi + \int \ps{u_i-u}{\nabla u_i\cdot \nabla \phi}=\\
 =\int \ps{u_i-u}{A(u_i)(\nabla u_i,\nabla u_i)+f_i}\phi + \int \ps{u_i-u}{\nabla u_i\cdot \nabla \phi}\, .
\end{gather}
Since $\norm{u_i-u}_{L^\infty(\B 1 0)}\to 0$, it is easy to see that this integral converges to $0$ as well. This completes the proof of the strong $H^1$ convergence.

As for the equation solved by $u$, we have for all smooth test functions $\phi\in C^\infty_C(\B 1 0)$:
\begin{gather}
 -\int \phi \Delta(u) = \int \ps{\nabla u}{\nabla \phi}=\lim_i \int \ps{\nabla u_i}{\nabla \phi} = \lim_i \int \phi \Delta (u_i)=\\
 = \lim_i \int \phi\qua{A(u_i)(\nabla u_i,\nabla u_i)+f_i} = \int \phi f +\lim_i \int \phi A(u_i)(\nabla u_i,\nabla u_i)\, .
\end{gather}
Since $A$ is continuous, and since $u_i\to u$ in $C^{\alpha/2}(\B {1}{0})$, we have $\abs{A(u_i)-A(u)}\to 0$. Here we interpret $A(\cdot)$ as a continuous function on $N$ which takes $\R^m\times \R^m$ into $\R^m$ in a bilinear way. Moreover, the strong convergence of $\nabla u_i$ to $\nabla u$ allow us to estimate
\begin{gather}
 \int \phi A(u_i)(\nabla u_i,\nabla u_i)=\int \phi \cur{\qua{A(u_i)-A(u)}(\nabla u_i,\nabla u_i) +A(u)(\nabla u_i,\nabla u_i)-A(u)(\nabla u,\nabla u)+A(u)(\nabla u,\nabla u)}=\\
 =\int \phi A(u)(\nabla u,\nabla u) + \norm{A(u_i)-A(u)}_\infty \int \phi \abs{\nabla u_i}^2+\int \phi \qua{A(u)(\nabla (u_i-u),\nabla u_i) +A(u)(\nabla u,\nabla (u_i-u))}\, .
\end{gather}
The strong convergence of $\nabla u_i$ to $\nabla u$ implies the thesis.
\end{proof}

Now we are in a position to prove the original proposition.

\begin{proof}[Proof of Proposition \ref{prop_weaklim}]
 The proof is based on the monotonicity formula for $u_i$ and the $\epsilon$-regularity theorem.
 
 Define the set
 \begin{gather}
  \Sigma=\bigcap_{r>0} \cur{x\in \overline{\B 1 0}\ \ s.t. \ \ \liminf_{i\to \infty} \theta_{u_i}(x,r) \geq \epsilon_0}\, ,
 \end{gather}
 where $\epsilon_0$ is taken from theorem \ref{th_app_eps_reg}. It is easy to see that this set is closed. Let $\B {r_j}{x_i}$ be a sequence of disjoint balls contained in $\B 3 0$ such that $x_j\in \Sigma$. For each $j$, there exists a subsequence of $u_i$ (still denoted with the same indexes for the sake of simplicity) such that
\begin{gather}\label{eq_u_ijk}
 \theta_{u_i}(x_j,r_j)\geq \epsilon_0>0\, .
\end{gather}
By a diagonal procedure, it is possible to find a subsequence of $u_i$ such that \eqref{eq_u_ijk} is valid for all $i$ and $j$. This implies immediately that
\begin{gather}
 \int_{\B 3 0}\abs{\nabla u_i}^2 \geq \sum_j r_j^{n-2}\theta_{u_i}(x_j,r_j) \geq \epsilon_0\sum_j r_j^{n-2}\, ,
\end{gather}
as desired. Note that evidently this uniform packing estimate implies upper estimates on the $n-2$ Minkowski content and Hausdorff measure.

For all $x\not \in \Sigma$, we can apply Lemma \ref{lemma_conv}, and obtain that $u$ is a \hol continuous function in a neighborhood of $x$ solving $\Delta u = A(u)(\nabla u,\nabla u)$. By standard estimates on continuous harmonic maps (see for example \cite[theorem 3.1]{moser_appr_harm}), $u$ is smooth in a neighborhood of $x$.

The last thing to check is that $u$ is globally weakly harmonic, and this is a consequence of the fact that $u$ is smooth and harmonic on $\Sigma^C$, and $\Sigma$ has bounded $n-2$ packing estimates. This fact implies that $\Sigma$ has $2$-capacity zero, and in particular for all compact sets $K\Subset \B 1 0$, there exists a sequence of smooth functions $\phi_i$ such that
\begin{gather}
 \phi_i(x) =1 \ \ \text{for } \ x \ \text{ in a neighborhood of }\ \Sigma\cap K\, , \quad \phi_i(x)=0 \ \ \text{ if } \ d(x,\Sigma\cap K)> i^{-1}\, , \quad \lim_{i\to \infty} \int_{\B 3 0} \abs{\nabla \phi_i}^2 =0\, .
\end{gather}
Since $u$ is smooth on $\Sigma^C$, we can write for all $\psi\in C^\infty_C(\B 1 0)$ that
\begin{gather}
 \int \psi \Delta(u) = \int \psi \phi_i \Delta(u) +\int \psi (1-\phi_i)\Delta(u)= \int \psi \phi_i \Delta(u) + \int \psi (1-\phi_i)A(u)(\nabla u,\nabla u)\, .
\end{gather}
As $i$ converges to $\infty$, $\phi_i$ converges to $1$ a.e. in $\B 3 0$. Thus by dominated convergence
\begin{gather}
 \lim_{i\to \infty} \int \psi (1-\phi_i) A(u)(\nabla u,\nabla u)=\int \psi A(u)(\nabla u,\nabla u)\, .
\end{gather}
Moreover, we have
\begin{gather}
 \limsup_{i\to \infty} \abs{\int \psi \phi_i \Delta(u)} \leq \limsup_{i\to \infty} \qua{\int \phi_i \abs{\nabla \psi \cdot \nabla u} + \int \psi \abs{\nabla \psi_i \cdot \nabla u}} =0\, .
\end{gather}

We are left to prove the unique continuation property. Note that if $u$ and $v$ are smooth, then unique continuation follows because $\Delta(u)=A(u)(\nabla u,\nabla u)=A(u)(\nabla u) [\nabla u]$ can be viewed as a linear equation on $u$ with smooth first order coefficient $A(u)(\nabla u)$.

Since $u$ and $v$ are assumed to be smooth outside the close set $\Sigma\cup \Sigma'$, and this close set is $n-2$ dimensional and thus non-disconnecting, we easily obtain that $u=v$ a.e. on $(\Sigma\cup\Sigma')^C$, and thus on the whole domain.

\end{proof}

\section{Main theorems}\label{s:main_theorems}
Now we are in a position to state precisely the main results we intend to obtain. The main theorem we want to prove is 
\begin{theorem}\label{th_mink}
Let $u:\B 3 0 \subseteq \R^m \to N$ an approximately harmonic mapping solving \eqref{eq_app1} and \eqref{eq_app_stat} with \eqref{eq_f} such that $3^{2-m}\int_{\B 3 0}\abs{\nabla u}^2\leq \Lambda$. Then for each $\epsilon>0$ there exists $C_\epsilon(m,\KN,\Lambda,\cF,\gamma,\epsilon)$ such that for all $r\in (0,1]$:
\begin{gather}\label{eq_mink}
\Vol\ton{\B r {\cS^k_{\epsilon,r}(u) } \cap \B 1 0 }\leq C_\epsilon r^{n-k}\, .
\end{gather}
As a corollary, we can estimate for all $r\in (0,1]$:
\begin{gather}
 \Vol\ton{\B r {\cS^k_{\epsilon}(u) } \cap \B 1 0 }\leq C_\epsilon r^{n-k}\, ,
\end{gather}
moreover $\cS^k_\epsilon$ is $k$-rectifiable.
\end{theorem}


As a corollary of this theorem we obtain the rectifiability of the strata.

\begin{theorem}\label{th_rect}
Let $u:\B 3 0 \subseteq \R^m \to N$ an approximately harmonic mapping solving \eqref{eq_app1} and \eqref{eq_app_stat} with \eqref{eq_f}. Then for all $k$ the strata $\cS^k(u)$ are $k$-rectifiable.
\end{theorem}

\vspace{5mm}

It is worth noticing that simple adaptation of the proofs described here allow to obtain a slightly better result. In particular, we can obtain uniform packing estimates instead of Minkowski estimates. Since the proof of this result requires no additional idea, but would make the exposition more technical and confusing, we state the result here and leave the details to the reader.
\begin{theorem}\label{th_pack}
Let $u:\B 3 0 \subseteq \R^m \to N$ an approximately harmonic mapping solving \eqref{eq_app1} and \eqref{eq_app_stat} with \eqref{eq_f} such that $3^{2-m}\int_{\B 3 0}\abs{\nabla u}^2\leq \Lambda$. Let $\cur{\B {r_x}{x}}_{x\in \cC}$ be a collection of pairwise disjoint balls with $r_x\in (0,1]$ and for all $x$, $x\in \cS^k_{\epsilon,r_x}$. Then there exists a constant $C_\epsilon(m,\KN,\Lambda,\cF,\gamma,\epsilon)$ such that
\begin{gather}
 \sum_{x\in \cC} r_x^n \leq C_\epsilon\, .
\end{gather}

\end{theorem}

\vspace{1cm}

In the following sections, we will develop all the techniques needed for the proof of these results. First of all, we will give a quantitative link between $0$-symmetry (and higher order symmetries) and the properties $\hat \theta$. We will then briefly recall without proof the Reifenberg theorems which we will use, and then we will turn to the $L^2$-best subspace approximation theorem and the covering arguments needed to complete the proofs.

\section{\texorpdfstring{Quantitative $\epsilon$-regularity theorems}{Quantitative epsilon-regularity theorems}}
The aim of this section is to show that there's a quantitative link between $\hat \theta$ and the almost symmetries of the map $u$. First, we are going to show an adaptation of \cite[theorem 3.3]{ChNa2}. In particular, we will show that if $\hat \theta(x,\cdot)$ is sufficiently \textit{pinched} on two consecutive scales (i.e., if $\hat\theta(x,r)-\hat\theta(x,r/2)$ is small enough), then $\B {r}{x}$ will be $(0,\delta)$-symmetric.

We will then turn our attention to higher order symmetries. We will show a very natural sufficient condition for $\B {r}{x}$ to be $(k,\delta)$-symmetric based on the geometry of the ``pinched points'' $y\in \B r x$ such that $\hat \theta(y,r)-\hat \theta(y,r/2)$ is small. Note that while a single pinched point is enough to guarantee $0$-symmetries, we will ask the set of pinched points to ``effectively span'' some $k$-dimensional affine subspace in order to guarantee higher order symmetry.

\subsection{Quantitative homogeneity}

It is easy to see that if $u$ is stationary harmonic and $\hat \theta(0,1)=\hat \theta(0,1/2)$, then $u$ must be $0$-symmetric. This is a direct consequence of \eqref{eq_theta'} and the unique continuation property for harmonic maps. By an easy compactness argument, we can see that this characterization of $0$-symmetric map has a rigidity property, in the sense that if $\hat \theta(0,1)-\hat \theta(0,1/2)$ is small enough, then $u$ is close to a $0$-symmetric map. In the case of an approximately harmonic map, this statement remains true up to focusing on a small enough scale, so that the error coming from $f$ becomes small enough.

\begin{remark}
 Throughout this section, we will assume that $u\in H^1(\B 3 0, N)$ is an approximately harmonic map satisfying \eqref{eq_app1} and \eqref{eq_app_stat} with \eqref{eq_f}. Moreover, we will also assume the uniform energy bound $\hat \theta(0,3)\leq \Lambda$.
\end{remark}

\begin{proposition}\label{prop_1pinch}
 Let $u\in H^1(\B 3 0)$ be a solution to \eqref{eq_app1} and \eqref{eq_app_stat} with \eqref{eq_f} and $\hat \theta(0,3)\leq \Lambda$. Then for all $\delta_1>0$, there exist $\delta_2=\delta_2(m,\Lambda,\gamma,\delta_1)>0$ such that if $\cF\leq \delta_2$ and for some $x\in \B 1 0$
 \begin{gather}\label{eq_thetapinch_1}
  \hat \theta(x,r)-\hat \theta(x,r/2)<\delta_2\, ,
 \end{gather}
then $\B r x$ is $(0,\delta_1)$-symmetric.
\end{proposition}
\begin{proof}
Consider by contradiction a sequence of maps $u_i$ and a sequence of balls $\B{r_i}{x_i}$ such that $\hat \theta (x_i,r_i)-\hat\theta(x_i,r_i/2)<i^{-1}$, $\cF<i^{-1}$, but such that the balls $\B {r_i}{x_i}$ are not $(0,\delta_1)$-symmetric. Let $T_i(y)=u(x_i+r_i(y))$ be their blow-up maps, and recall that $\hat \theta_{T_i}(0,s)=\hat \theta_{u_i}(x_i,sr)$. By \eqref{eq_theta_unif}, $T_i$ have uniform $H^1$ bounds, and so there exists a weakly convergence subsequence, for convenience denoted with the same index. Note that $T_i$ are approximately harmonic maps solving
\begin{gather}
 \Delta(T_i)= A(T_i)(\nabla T_i,\nabla T_i)+\tilde f_i\, , \quad \tilde f_i(y)= r_i^2 f_i(x_i+r_iy)\, .
\end{gather}
By \eqref{eq_f}, we have $\int_{\B 1 0} \abs{\tilde f_i}^2 = r^{4-m}\int_{\B {r_i}{x_i}} \abs f^2 \leq i^{-1} r_i^\gamma \to 0$. Thus $T$ is weakly harmonic by Proposition \ref{prop_weaklim} and smooth away from a close set $\Sigma$ of dimension $n-2$. 

Moreover, by \eqref{eq_hat'}, we have the estimate
\begin{gather}
 \int_{\B 1 0 \setminus \B {1/2}{0}} \abs{y\cdot \nabla T_i}^2 \leq c(m) \qua{\hat \theta_{T_i}(0,1)-\hat \theta_{T_i}(0,1/2)}\to 0\, .
\end{gather}
Thus $T$ is radially invariant on $\B{1}0\setminus \B {1/2}{0}$, and by unique continuation it is homogeneous. Indeed, let $T'$ be the homogeneous continuation of $T$ over the whole $\B 1 0$. It is easily seen that both maps are weakly harmonic and smooth away from an $n-2$ dimensional set, thus we can apply point \eqref{it_uc} of proposition \ref{prop_weaklim} and prove that $T=T'$.  Since $T_i$ converges weakly in $H^1$ to $T$, and strongly in the $L^2$ norm, we have reached a contradiction.
\end{proof}

Note that, as a corollary of the proof, we obtain a characterization of all tangent maps for approximate harmonic function. 
\begin{corollary}\label{cor_weak_harm_limit}
 Let $u$ be as above, then all of its tangent maps are weakly harmonic homogeneous maps. In particular, for all $x\in \B 1 0$, and for all sequences $r_i\to 0$, there exists a subsequence $r_{i_j}$ such that $T_j = T^u_{x,r_{i_j}}$ converges in the weak $H^1_{loc}(\R^m,N)$ sense to a $0$-symmetric weakly harmonic map which is smooth on an open, dense, connected subset.
\end{corollary}

We close this section with the characterization of $\cS^k(u)$ promised above.
\begin{lemma}\label{lemma_Sk}
 Let $u$ be an approximate harmonic map with \eqref{eq_f}. Then
 \begin{gather}
  \cS^k(u) = \bigcup_{\epsilon>0}\cS^k_{\epsilon}=\bigcup_{\epsilon>0} \bigcap_{r>0}\cS^k_{\epsilon,r} = \cur{x\in \Omega \ \ s.t. \ \ \text{no tangent maps at } x \text{ is } k+1-\text{symmetric}}\, .
 \end{gather}
\end{lemma}
\begin{proof}
 Let $x\in \Omega$. We now know that for all $x$, there exists at least one (possibly more) tangent map $T$, and all tangent maps are $0$-symmetric.
 
 Suppose that at $x\in \cS^k_\epsilon$ for some $\epsilon>0$. Then since for all $r>0$, $T^u_{x,r}$ is at least $\epsilon$-apart in the $L^2$ sense from all $k+1$-symmetric maps, then also any tangent map $T=\lim_{i} T^u_{x,r_i}$ must be $\epsilon$-apart from all $k+1$ symmetric maps. Thus
 \begin{gather}
  \bigcup_{\epsilon>0}\cS^k_{\epsilon}\subseteq \cur{x\in \Omega \ \ s.t. \ \ \text{no tangent maps at } x \text{ is } k+1-\text{symmetric}}\, .
 \end{gather}
Now if $x\not \in \bigcup_{\epsilon>0} \cS^k_\epsilon$, by definition for all $i>0$, there exists some $r_i>0$ and some $k+1$-symmetric map $T_i$ such that
\begin{gather}
 \int_{\B 1 0} \abs{T^u_{x,r_i} - T_i}^2 \leq i^{-1}\, .
\end{gather}
By passing to a subsequence $T_i\to T$ which is $k+1$-symmetric, we see also that $T^u_{x,r_i} \wto T$. If there's a subsequence of $r_i$ converging to $0$, then $T$ is by definition a tangent map, and it is also $k+1$-symmetric. If $r_i$ is bounded away from $0$, then $u=T$ on a ball of positive size around $x$, and in particular all tangent maps of $u$ at $x$ are equal to $T$.  In either case, we have proved the claim.
\end{proof}

\subsection{Quantitative higher order symmetries}
In order to have higher order symmetries, one point where $\hat \theta$ is pinched is not enough. However, if we have more points where $\hat \theta$ is pinched, and these points span in some sense a $k$-dimensional space, this is enough to guarantee higher order symmetries.
\begin{example}
 As a guiding example, consider again the case of a stationary harmonic map. If for two distinct points $x_1,x_2$ we have $\theta(x_i,1)-\theta(x_i,1/2)=0$, then the map is homogeneous with respect to $x$ and homogeneous with respect to $y$, which immediately implies that this map is invariant with respect to the line $L$ joining $x$ and $y$.
 
 Moreover, note also the following. If we take any $z\not \in L$, then at $z$ two distinct directional derivatives of $u$ are null, one in the $L$ direction, and another in the direction joining $z$ and $x$ (or $y$). Now consider a small enough ball $\B r z$. On this ball, two directional derivatives are null, and if $r$ is small enough this ball is almost $0$-symmetric by monotonicity of $\theta(z,\cdot)$. Thus this ball will be almost $2$-symmetric.
\end{example}

 In the next lemmas, we will prove in detail quantitative versions of these observations.  Before doing that, let us record the definition of a quantitative version of linear independence, which will be used throughout the rest of this section.

\begin{definition}
Let $y_0,\cdots,y_k\subset \B 1 0\subset \R^m$. We say that these points $\rho$-effectively span a $k$-dimensional affine subspace if for all $i=1,\cdots,k$, 
\begin{gather}
 y_i\not \in \B {2\rho} {y_0+\operatorname{span}(y_1-y_0,\cdots,y_{i-1}-y_0)}\, .
\end{gather}
As an immediate consequence, we obtain that $\cur{y_i}$ are effectively linear independent, in the sense that for all
\begin{gather}
x\in \cur{y_0+\operatorname{span}(y_1-y_0,\cdots,y_{k}-y_0)}\, , 
\end{gather}
there exists a unique set of numbers $\cur{\alpha_i}_{i=1}^k$ such that
\begin{gather}
 x=y_0+\sum_{i=1}^k \alpha_i (y_i-y_0)\, , \quad \abs{\alpha_i}\leq C(m,\rho)\norm{x-y_0}\, .
\end{gather}
Moreover, note that this property is preserved under limits, as opposed to the property of being simply linearly independent. Indeed, if we have a sequence $\cur{y_{i,j}}_{i=0}^k$ such that for all $j$, $\cur{y_{i,j}}_{i=0}^k$ $\rho$-effectively span a $k$-dimensional space, and if $\lim_j y_{i,j}=\bar y_i$, then also $\cur{\bar y_i}_{i=0}^k$ $\rho$-effectively spans a $k$-dimensional space. 
\end{definition}

\begin{definition}
 Given $F\subset \B 1 0$, we say that $F$ $\rho$-effectively spans a $k$-dimensional subspace if there exist $\cur{y_0,\cdots,y_k}\subseteq F$ that $\rho$-effectively spans a $k$-dimensional subspace according to the previous definition.
\end{definition}

%
%
With this concept in mind, we can extend Proposition \ref{prop_1pinch} to the case where we have $k+1$ distinct points of pinching for $\hat \theta$. Note that we will not actually use this proposition, indeed we will need the more refined version of this statement, which is the main focus of Section \ref{sec_dist_est}. However, we think it is reasonable to record this proposition in order to give the reader a better understanding of the direction that our argument is taking.
\begin{proposition}\label{prop_Kpinch}
 Let $u\in H^1(\B 3 0)$ be a solution to \eqref{eq_app1} and \eqref{eq_app_stat} with \eqref{eq_f} and $\theta(0,3)\leq \Lambda$. Then for all $\epsilon,\rho>0$, there exist $\delta(m,\Lambda,\gamma,\epsilon,\rho)$ such that if $\cF<\delta$ and for some $\cur{x_i}_{i=0}^k\subset \B 1 0$ we have
 \begin{enumerate}
  \item $\cur{x_i}$ $\rho$-effectively spans a $k$-dimensional affine subspace $V$,
  \item $\hat \theta(x_i,r)-\hat \theta(x_i,r/2)<\delta$ for all $i$,
 \end{enumerate}
then $\B r x$ is $(k,\epsilon)$-symmetric.
\end{proposition}
\begin{proof}
 The proof is a simple adaptation of Proposition \ref{prop_1pinch}.
\end{proof}


The next proposition shows that in case when the set of points with pinching effectively spans a $k$-dimensional plane $V$, then $\cS^k_{\epsilon,r}$ is contained in a tube around $V$.

\begin{proposition}\label{prop_k+1_pinch}
 Let $\rho,\epsilon>0$ be fixed. There exists $\delta_3(m,\Lambda,\KN,\gamma,\rho,\epsilon)>0$ such that the following holds.
 Let $F=\cur{y\in \B 2 0 \ \ s.t. \ \ \hat \theta(y,1)-\hat \theta(y,\rho)<\delta_3}$. If $\cF\leq \delta_3$ and $F$ $\rho$-effectively spans a $k$-dimensional subspace $V$, then 
 \begin{gather}
  \cS^k_{\epsilon,\delta_3}\subseteq \B {2 \rho} {V}\, .
 \end{gather}
\end{proposition}
We split the proof of this proposition into two parts, the first of which is contained in the following technical lemma.
\begin{lemma}\label{lemma_k-span}
 Let $u$ be as above. There exits a $\delta_4(m,\Lambda,\KN,\gamma,\rho,\epsilon)>0$ such that if $\cF<\delta_4$ and 
 \begin{gather}\label{eq_P_pinch}
  \int_{\B 1 0} \abs{P\cdot \nabla u}^2<\delta_4\, ,
 \end{gather}
 for some $k+1$ dimensional linear subspace $P$, then $\cS^k_{\epsilon,\bar r}\cap \B{1/2} {0}=\emptyset$, where $\bar r = \delta_4^{\frac 1 {2 (m-2)}}$.
\end{lemma}
\begin{remark}
Notationally we define $|P\cdot \nabla u|^2\equiv \sum |\nabla_{e_i}u|^2$, where $e_i$ is an orthonormal basis of $P$.	
\end{remark}

\begin{proof}
 We want to show by contradiction that for all $x\in \B{1/2}{0}$, there exists an $r_x \in [\bar r,1/2]$ such that $\B {r_x}{x}$ is $(k+1,\epsilon)$-symmetric.  Note that, by monotonicity, for all $x\in \B {1/2}{0}$ there exists an $r_x\in [\bar r,1/2]$ such that
 \begin{gather}\label{eq_k+1_1}
  \hat \theta(x,r_x)-\hat \theta(x,r_x/2)<\frac{C_1(m,\gamma)\Lambda}{-\log(\delta_4)}\, .  
 \end{gather}
Indeed, otherwise we would have
 \begin{gather}
  c(m,\gamma)(\Lambda +\delta_4) \stackrel{\eqref{eq_theta_unif}}{\geq} \hat \theta(x,1/2)=\sum_{i=1}^{-\log(\bar r)+1} \ton{\hat \theta(x,2^{-i})-\hat \theta (x,2^{-i-1})} \stackrel{\eqref{eq_hatbounds}}{\geq} c(m) C_1(m,\gamma)\Lambda \, ,
 \end{gather}
 which is impossible if we set $C_1(m,\gamma)=2 c(m,\gamma)c(m)^{-1}$.
 Moreover, note that
 \begin{gather}\label{eq_k+1_2}
  r_x^{2-n} \int_{\B {r_x}{x}} \abs{P\cdot \nabla u}^2 \leq \delta_4 r_x^{2-n}\leq \delta_4^{1/2}\, 
 \end{gather}
by definition of $r_x$. Thus consider a contradicting sequence $u_i$ with $\delta_{4,i}\to 0$ such that there exists $x_i\in \B{1/2}{0}$ and $r_i\in [\bar r,1]$ such that $\B {r_i}{x_i}$ is not $(k+1,\epsilon)$-symmetric but such that \eqref{eq_k+1_1} and \eqref{eq_k+1_2} are valid. By a simple rotation, we can assume that the $k+1$ dimensional subspace $P$ is fixed throughout the sequence. Consider the maps $T_i= T^{u_i}_{x_i,r_i}$. Their weak limit converges to some weakly harmonic $T$ which is $0$-symmetric by unique continuation, \eqref{eq_k+1_1} and Proposition \ref{prop_1pinch}. Moreover, $T$ is also invariant wrt the $k+1$-dimensional $P$ by \eqref{eq_k+1_2}. This clearly is a contradiction.
 
\end{proof}

Now we turn to the proof of the main proposition.
\begin{proof}[Proof of Proposition \ref{prop_k+1_pinch}]
 Let $\cur{y_0,\cdots,y_k}\subseteq F$ $\rho$-effectively span the $k$-dimensional subspace $V$, and consider any $x\in \B 2 0 \setminus \B {2\rho}{V}$. Note that for all $i=0,\cdots,k$, we have
 \begin{gather}
  \B {\rho}{x}\subset \B {2} {y_i}\setminus \B {\rho} {y_i}\, .
 \end{gather}
By \eqref{eq_hat'}, we obtain that for all $i=0,\cdots,k$:
\begin{gather}
 \int_{\B \rho {x}} \abs{(z-y_i)\cdot \nabla u}^2(z)   \leq c(m)\qua{ \hat \theta(y_i,1)-\hat \theta(y_i,\rho)}\, .
\end{gather}
Let $w$ be any norm $1$-vector in $\hat V$, the linear subspace associated to the affine $V$. By definition of $\cur{y_i}$, there exists $\cur{\alpha_i}_{i=1}^k$ such that
\begin{gather}
 w=\sum_i \alpha_i (y_i-y_0) \ \ \quad \abs{\alpha_i}\leq c(m,\rho)\, .
\end{gather}
Thus we also have
\begin{gather}
 \int_{\B \rho {x}} \abs{w\cdot \nabla u}^2(z)   \leq C(m,\rho) \sum_i \int_{\B \rho {x}} \abs{[(y_i-z)+(z-y_0)]\cdot \nabla u}^2(z)   \leq\\
 \leq c(m,\rho)\sum_i \qua{\hat \theta(y_i,1)-\hat \theta(y_i,\rho)}\leq c(m,\rho,\gamma)\delta_3\, .
\end{gather}
This in particular implies
\begin{gather}
 \int_{\B \rho {x}} \abs{V\cdot \nabla u}^2   \leq c(m,\rho,\gamma)\delta_3\, .
\end{gather}
In order to gain one more independent direction along which the energy is small, set for all $z\in \B 1 0\setminus V$
\begin{gather}
 h(z)= \frac{z-\pi_V(z)}{\norm{z-\pi_V(z)}}\, .
\end{gather}
By an argument similar to the one above, we obtain also
\begin{gather}
 \int_{\B \rho {x}} \abs{h(z)\cdot \nabla u}^2(z)  \leq c(m,\rho,\gamma)\delta_3 \, .
\end{gather}
By a geometric argument, it is easy to see that if $d(z,V)\geq \rho$, then $\abs{h(z)-h(x)}\leq C(m) \abs{x-z} \rho^{-1}$. This implies that for all $r\leq \rho$, we have
\begin{gather}
 \frac 1 2 \int_{\B r {x}} \abs{h(x)\cdot \nabla u}^2(z)   \leq \int_{\B r {x}} \abs{h(z)\cdot \nabla u}^2(z)+\int_{\B r {x}} \abs{h(x)-h(z)}^2\abs{\nabla u}^2(z)  \leq c(m,\rho,\gamma)\delta_3 + c(m,\rho,\gamma,\Lambda) r^m\, ,
\end{gather}
where in the last line we have used the uniform bound $\theta(x,r)\leq c(m,\gamma)(\Lambda+1)$ for all $x\in \B 2 0$ and $r\in [0,1]$ (see the estimates in \eqref{eq_theta_unif}).

For the $k+1$-dimensional linear subspace $P=\hat L \oplus h(x)$, and for all $r\leq \rho$, we get the estimate
\begin{gather}
  r^{2-m}\int_{\B {r} {x}} \abs{P\cdot \nabla u}^2  \leq c(m,\rho,\gamma)\delta_3 r^{2-m} + C(m,\rho,\gamma,\Lambda) r^2\, .
\end{gather}
Now we can choose $r$ and subsequently $0<\delta_3<<\delta_4$ small enough in order to apply Lemma \ref{lemma_k-span} to $\B r x$, and we obtain the thesis.
\end{proof}

We close this section by observing that if $\B r x $ is not $(k+1,\epsilon)$-symmetric, but it is almost $0$-symmetric, then necessarily $u$ must have some energy on \textit{any} $k+1$ distinct directions on $\B r x$. Actually, part of this energy must be concentrated on the annulus $\B r x\setminus \B {r/2}{x}$.
\begin{lemma}\label{lemma_lower_bound}
Let $u$ be as above. Then for each $\epsilon>0$ there exists $\delta_5(m,\KN,\Lambda,\gamma,\epsilon)>0$ such that if $\cF<\delta_5$, and $\B 1 0$ is $(0,\delta_5)$-symmetric but is {\it not} $(k+1,\epsilon)$-symmetric, then for every $k+1$ linear subspace $P$ we have 
\begin{align}
\int_{A_{3/8,1/2}(0)} |P\cdot \nabla u|^2 \geq \delta_5\, ,
\end{align}
where $A_{3/8,1/2}(0)\equiv \B {1/2} 0 \setminus \B {3/8}{0}$.
\end{lemma}

\begin{proof}

Also this proof is based on a simple contradiction argument which hinges on the $H^1_{weak}$ compactness given by the uniform energy bounds.  Thus let $u_i$ be a contradicting sequence. In particular, let $u_i$ be approximately harmonic maps with $\cF_i\leq i^{-1}$ such that $\B 1 0$ is $(0,i^{-1})$-symmetric but not $(k+1,\epsilon)$-symmetric. Moreover, let $P$ be a $k+1$-dimensional subspace such that for all $i$
\begin{align}\label{e:best_subspace_lower_energy:1}
\fint_{A_{3/8,1/2}(0)}|P\cdot \nabla u_i|^2 \leq i^{-1}\to 0\, .
\end{align}
Note that we can assume that $P$ is not changing with $i$ simply by making a rotation in the domain space.
After passing to a subsequence, $u_i\wto u$ in the weak $H^1$ sense, with $u$ being a weakly harmonic map. Now \eqref{e:best_subspace_lower_energy:1} and the $H^1_{weak}$ convergence guarantee that 
\begin{gather}
 \int_{A_{3/8,1/2}(0)}\abs{P\cdot \nabla u}^2 = 0\, .
\end{gather}
Moreover, $u$ will be $0$-homogeneous by unique continuation and the pinching assumption. Thus, we obtain that 
\begin{gather}
 \int_{\B 1 0 }\abs{P\cdot \nabla u}^2 = 0\, ,
\end{gather}
and this implies that $u$ is $k+1$-symmetric.  Since $u_i$ converges strongly to $u$ in $L^2(\B 1 0)$, we obtain a contradiction. 
\end{proof}

\subsection{Uniformity of the energy and of the non-symmetry}
One moral to be taken from the previous section, in particular from Proposition \ref{prop_Kpinch}, is that if we have a lot of pinched points which span something $k$-dimensional, then the map $u$ is almost constant along these $k$-directions.

Here we prove another two important variations of this general philosophical point. First of all, we will show that in the situation described above, $u$ is almost constant also in some $H^1$ sense, not just in an $L^2$ sense. In particular $\hat \theta$ remains almost constant on all pinched points.

\begin{lemma}\label{lemma_unipinch}
  Let $u:\B 3 0 \to N$ be a solution to \eqref{eq_app1} and \eqref{eq_app_stat} with \eqref{eq_f}. Let $\rho>0$ and $\eta>0$ be fixed, and assume that for all $y\in \B 1 0$, $\hat \theta(y,1)\leq E$, then there exists $\delta_6=\delta_6(m,\Lambda,\KN,\rho,\gamma,\eta)$ such that if $\cF\leq \delta_6$, the following holds.
 Let $F(u,\delta_6)\subseteq\cur{y\in \B 1 0 \ \ s.t. \ \ \hat \theta(y,\rho)>E-\delta_6}$. If $F$ $\rho$-effectively spans a $k$-dimensional subspace $L$, then for all $x\in L\cap \B 2 0$, we have
 \begin{gather}
  \hat \theta(x, \rho)>E-\eta\, .
 \end{gather}
 Moreover, if $k\geq m-1$, then $E\leq \eta$.
\end{lemma}
\begin{proof}
 We prove this statement by contradiction. Fix any $\eta>0$, and let $u_i$ be a contradicting sequence. 
 In particular, $u_i$ will be approximately harmonic maps on $\B 3 0$ with $\theta_{u_i}(y,1)\leq E$ for all $i$, and such that $F_i=F(u_i,i^{-1})$ $\rho$-effectively spans a $k$-dimensional subspace $L$. Up to translations and rotations, we can assume that $L$ is fixed for all $i$. 
 
 Moreover, we assume by contradiction that there exists some $x_i\in L$ such that $\hat \theta_{u_i}(x_i,\rho)\leq E-\epsilon$. Note that, evidently, we can assume that $L$ has dimension at least $1$, otherwise there's nothing to prove.
 
 By weak compactness, we can pass to a subsequence and have $u_i\wto u$, where $u$ is a weakly harmonic map, and
 \begin{gather}
  \abs{\nabla u_i}^2 d\Vol \wto \mu \equiv \abs{\nabla u}^2 d\Vol + \nu \, ,
 \end{gather}
where the convergence is in the weak sense of measures, and $\nu$ is the defect measure, which is nonnegative by Fatou's lemma. Set by definition $\theta_\mu(y,r)=r^{2-n}\mu(\B r y)$, and note that $\theta_\mu$ is monotone in $r$ for all $y$ fixed. Indeed, let $0<r_1<r_2$, and consider that
\begin{gather}
 \theta_{\mu}(y,r_2)-\theta_\mu (y,r_1)=\lim_{i\to \infty} \qua{\theta_{u_i}(y,r_2)-\theta_{u_i}(y,r_1)} \stackrel{\eqref{eq_hatbounds}}{=}\lim_{i\to \infty} \qua{\hat \theta_{u_i}(y,r)-\hat \theta_{u_i}(y,r_1)} {\geq} 0
\end{gather}
Now let $\cur{y_{i,j}}_{j=0}^k\subset F_i$ be a set of points which $\rho$-effectively spans $L$. By passing to a subsequence, we can assume that $x_i\to x\in L$ and $\lim_{i\to \infty} y_{i,j}=\bar y_j$, where $\cur{\bar y_j}$ $\rho$-effectively spans $L$.

For all $j$ and $r>0$, we have that
\begin{gather}\label{eq_y_i2y}
 \ton{\frac{r-\abs{y_i-y}}{r}}^{n-2}\theta_{u_i}(y_i,r-\abs{y_i-y})\leq \theta_{u_i} (y,r)\leq \ton{\frac{r}{r+\abs{y_i-y}}}^{n-2}\theta_{u_i}(y_i,r+\abs{y_i-y})\, ,
\end{gather}
thus by taking the limit on $i$, we can conclude that for all $j$:
\begin{gather}
 \theta_\mu(y_j,1)-\theta_\mu(y_j,\rho)=0\, .
\end{gather}

By an easy adaptation of \cite[lemma 1.7]{lin_stat} (in particular, by the proof of point $ii$ in this lemma, carried out at pages 797--800), this implies that $\mu$, $\nu$ and $u$ are radially invariant on $\B{1}{y_j}\setminus \B {\rho}{y_j}$ for all $j$. Since $y_j$ $\rho$-effectively span $L$, as an immediate consequence we have that $u$, $\nu$ and $\mu$ are invariant wrt $L$ on the whole $\B 2 0$.

Thus $\theta_\mu(y,\rho)=E$ for all $y\in L$, in particular $\theta_\mu(x,\rho)=E$. Since $x_i\to x$ and $\rho>0$, we obtain our contradiction by \eqref{eq_y_i2y}.

\vspace{3mm}

As a last point, if we assume that $L$ has dimension at least $m-1$, then we know that both $\nu$ and $u$ are invariant wrt an $m-1$ dimensional affine subspace. However, by point $2$ in Proposition \ref{prop_weaklim}, the support of $\nu$ must have finite $m-2$ dimensional Hausdorff measure. Thus, we see that $\nu=0$, which means that $u_i$ converges strongly in $H^1$ to $u$, which is an $m-1$-invariant weakly harmonic map. In other words, $u$ is a weakly harmonic map depending only on $1$ variable, and by standard Sobolev embeddings those maps are well-known to be continuous. Since $u$ is also $0$-symmetric wrt $y_0$, it is a constant map. This proves the last claim.

\end{proof}

We close this section with the last technical result we need. In particular, we prove that almost symmetry (or lack thereof) is preserved under some suitable pinching condition.

\begin{lemma}\label{lemma_ksym_pinch}
  Let $u:\B 8 0 \to N$ be a solution to \eqref{eq_app1} and \eqref{eq_app_stat} with \eqref{eq_f}. Let $\rho>0$ and $\epsilon>0$ be fixed. There exists $\delta_7=\delta_7(m,\Lambda,\KN,\rho,\gamma,\epsilon)$ such that if $\cF\leq \delta_7$, then the following holds.  If $ \hat \theta(0,1)-\hat \theta(0,1/2)<\delta_7$ and there exists a point $y\in \B 3 0$ with
  \begin{enumerate}
   \item $ \hat \theta(y,1)-\hat \theta(y,1/2)<\delta_7$,
   \item for some $r\in [\rho, 2]$, $\B r y$ is not $(k+1,\epsilon)$-symmetric
  \end{enumerate}
Then $\B r 0$ is not $(k+1,\epsilon/2)$-symmetric.
\end{lemma}
\begin{remark}
 Note that we do not assume that $\abs {y-0}\geq \rho$, and thus we cannot obtain in general that $u$ will be invariant on the line joining $0$ and $y$.
\end{remark}
\begin{proof}
 Consider again a contradicting sequence. In particular, $u_i$ will be a sequence of maps with $\hat \theta(0,1)-\hat \theta(0,1/2)\leq i^{-1}$, and for some $y_i\in \B 3 0$ we will have $\hat \theta(y_i,2)-\hat \theta(y_i,1/2)\leq i^{-1}$. Moreover, $\B {r}{y_i}$ is not $(k+1,\epsilon)$-symmetric, but $\B r 0$ is $(k+1,\epsilon/2)$-symmetric.

Thus, for all $i$, there exists an $h_i:\B {10}{0}\to N$ which is $k+1$-symmetric such that
 \begin{gather}
  \fint_{\B r 0} \abs{u_i-h_i}^2\leq \epsilon/2\, .
 \end{gather}

By passing to a subsequence if necessary, we obtain $y_i\to y\in \overline{\B 3 0}$, $u_i\wto u$ in $H^1(\B 8 0)$ and thus $u_i\to u$ in the strong $L^2(\B 8 0)$ sense, with $u$ weakly harmonic and homogeneous wrt $0$ and $y$. Moreover, $h_i\wto h$ in $L^2$, where $h$ is a $k+1$-symmetric map with
\begin{gather}
 \fint_{\B r 0} \abs{u-h}^2 \leq \limsup_{i\to \infty} \fint_{\B r 0} \abs{u-h_i}^2 \leq 2\epsilon/3\, .
\end{gather}

If $y=0$, then obviously we obtain a contradiction. Indeed, since $N$ is compact, we have
\begin{gather}
 \lim_{i\to \infty} \fint_{\B {r}{y_i}} \abs{u_i-h}^2 =\fint_{\B r 0} \abs{u-h}^2\leq 2\epsilon/3\, .
\end{gather}

In a similar way, if $y\neq 0$, then $u$ is invariant wrt the line joining $0$ and $y$, and we get
\begin{gather}
 \lim_{i\to \infty} \ton{\fint_{\B {r}{0}} \abs{u_i(y_i+z)-h(z)}^2 \ dz}^{1/2} \leq \lim_{i\to \infty} \ton{\fint_{\B {r}{0}} \abs{u_i(y_i+z)-u(y+z)}^2 \ dz}^{1/2} + (2\epsilon/3)^{1/2}=(2\epsilon/3)^{1/2}\, .
\end{gather}

\end{proof}
\vspace{.5cm}

\section{Reifenberg theorem}
In this section, we quote without proof the appropriate Reifenberg results from \cite{nava+} that are needed in this paper.  These results have since been extended and improved upon, with simplified proofs, in \cite{ENV_Reif}. Before doing that, we write a brief introduction to these results, and state the necessary definitions.

\subsection{Reifenberg theorem in literature}
The Reifenberg topological disk theorem, introduced in \cite{reif_orig}, states that if a subset $S\subset \R^n$ is sufficiently close in the Hausdorff sense at all scales to a $k$-dimensional plane, then $S$ is $C^{0,\alpha}$-homeomorphic to a disk. 

This theorem has been improved during the years, with the objective of obtaining some $C^{0,1}$ information on $S$. For example, Toro proves in \cite{toro_reif} that the correspondence is $C^{0,1}$ under some summability assumption on Jones' $\beta_2$ numbers for the set $S$. More general results along the same line are available in \cite{david_semmes,davidtoro}. 

It is worth mentioning that just by working with $\beta_2$-numbers, without the Reifenberg-type techniques, rectifiability results for sets and measures similar to the ones discussed have been obtained very recently in \cite{AzzTol,Tol}.

In \cite{nava+} and \cite{ENV_Reif}, effective estimates in the form of upper Ahlfor's regularity and rectifiability are obtained for sets and measures under the appropriate Dini conditions on the $\beta_2$-numbers.  These results play a key role in the finiteness and structure theorems of this paper.  Here we quote these theorems, and refer the reader to \cite{nava+} and \cite{ENV_Reif} for their proofs.



\subsection{Definitions}
Here we define the so-called Jones' $\beta_2$ numbers. 
\begin{definition}\label{deph_D}
Let $\mu$ be a nonnegative Radon measure on $\B 3 0$, and fix any $r>0$ and $k\in \N$. We define the $k$-dimensional Jones' $\beta$ number by
\begin{gather}\label{eq_deph_D}
\beta_{2,\mu}^k (x,r)^2 = \min_{V\subseteq \R^n}\int_{B_{r}(x)}\frac{d^2(y,V)}{r^2}\,\frac{d\mu(y)}{r^k}\, ,
\end{gather}
where the minimum is taken over all affine subspaces $V$ of dimension $k$.
\end{definition}

It is clear that $\beta_2$ is suitable to quantify how close the support of $\mu$ is to a $k$-dimensional subspace. Note that the scaling factor $r^{-2-k}$ in the definition of $\mu$ is chosen to make $\beta_2$ ``scale invariant'' in some sense. Indeed, $r^{-2}$ takes care of the scaling properties of $d(x,V)^2$, and since we expect $\mu$ to behave like a $k$-dimensional measure, $r^{-k}\mu(\B r x)$ is the right scale invariant quantity to consider.  This is the case if, for example, $\mu$ if $k$-\al regular, in the sense that for all $x\in \supp \mu\cap \B 1 0$ and $r\leq 1$
\begin{gather}
 C^{-1} r^k \leq \mu(\B r x) \leq C r^k\, ,
\end{gather}
or if $\mu$ is the $k$-dimensional Hausdorff measure on any set $S$.  It is worth recording two basic properties of $\beta$.
\begin{lemma}
 $\beta^2_{2}$ is monotone in $\mu$, in the sense that if $\mu\leq \nu$, then for all $x,r$
 \begin{gather}
  \beta^k_{2,\mu} (x,r)^2\leq \beta^k_{2,\nu} (x,r)^2\, .
 \end{gather}
Moreover, if $\abs {x-y}\leq r$, then
\begin{gather}
 \beta^k_{2,\mu}(x,r)^2 \leq 2^{k+2} \beta^k_{2,\mu}(y,2r)^2\, .
\end{gather}
\end{lemma}
\begin{proof}
 Both these properties are immediate consequences of the definition of $\beta$.
\end{proof}

\subsection{Generalized Reifenberg Theorems}
Now we are ready to state the two versions of the quantitative Reifenberg theorems from \cite{nava+} that we will use to prove the uniform volume bounds on $\cS^k_{\epsilon,r}$.

\begin{theorem}\cite[theorem 3.4]{nava+}\label{th_disc_reif}
For some constants $\delta_R(m)$ and $C_R(m)$ depending only on the dimension $m$, the following holds. Let $\{\B {r_x/5}{x}\}_{x\in \cC }\subseteq \B 3 0 \subset \R^m$ be a collection of pairwise disjoint balls with their centers $x\in \B 1 0$, and let $\mu\equiv \sum_{x\in \cC}\omega_k r_x^k \delta_{x}$ be the associated measure. Assume that for each $B_r(x)\subseteq B_2$
\begin{gather}\label{eq_reif_ass}
 \int_{\B r x )}\ton{\int_0^r \beta^k_{2,\mu}(y,s)^2 \,{\frac{ds}{s}}}\, d\mu(y)<\delta_R^2 r^{k}\, .
\end{gather}
Then we have the uniform estimates
\begin{gather}
\sum_{x\in \cC} r_x^k<C_R(m)\, .
\end{gather}
\end{theorem}
\begin{remark}
See \cite{ENV_Reif} for a more recent generalization of the above.	
\end{remark}

This theorem will be used on some carefully chosen discrete approximation of the singular set of $u$. In order to guarantee the assumption \eqref{eq_reif_ass}, we will use the $\beta_2$ estimates of Section \ref{sec_dist_est}.  The next result is the rectifiable Reifenberg from \cite{nava+}, see also \cite{ENV_Reif}:

\begin{theorem}\cite[theorem 3.3]{nava+}\label{th_rect_reif}
There exists $\delta_R(m)>0$ such that the following holds.  Let $S\subseteq \B 3 0 \subseteq \dR^n$ be a $\lambda^k$-measurable subset, and assume for each $\B r x$ with $x\in \B 1 0 $ and $r\leq 1$ we have
\begin{align}\label{eq_reif_ass_rect}
\int_{S\cap \B r x}\,\ton{\int_0^r \beta^k_{2,\lambda^k|_S}(y,s)^2\,\frac{ds}{s}}\, d\lambda^k(y)<&\delta_R^2r^{k}\, . 
\end{align}
Then $S$ is $k$-rectifiable such that for each $x\in S$ we have $\lambda^k(B_r(x))<C_R r^k$.
\end{theorem}
\begin{remark}
The basis for the ideas in the above are a technical refinement of a new $W^{1,p}$-Reifenberg which is proved by the authors in \cite{nava+}.  Since we do not directly need this $W^{1,p}$-Reifenberg we do not state it here.
\end{remark}
\begin{remark}
Note that \cite[theorem 1.1]{AzzTol} proved the above without the Ahlfor's upper bound.  See also \cite{ENV_Reif} for a more recent generalization, which both applies to a much more general class of measures, and does so under weaker assumptions.
\end{remark}
 
\vspace{.5cm}

\section{Proof of the main theorems}
In order to be in a position to prove the main theorem, we need to obtain two important ingredients. First we will discuss estimates linking the $\beta_2$ of a generic measure $\mu$ with support contained in $\cS^k_{\epsilon,r}(u)$ and the monotone quantity $\hat \theta$, and later on we will describe a covering argument that will allow us to split the covering of the set $\cS^k_{\epsilon,r}(u)$ into suitable pieces with nice estimates.

\subsection{\texorpdfstring{$L^2$ subspace approximation theorem}{L2 subspace approximation theorem}}\label{sec_dist_est}
The aim of this section is to prove that the $\beta_2$ defined in the previous sections can be controlled using the monotone quantities $\theta$ and $\hat \theta$, and the parameters $\cF$ and $\gamma$. 

In order to ease the notation, we define for $x\in \B 1 0$ and $r>0$ the quantity
\begin{align}
W_{r}(x) \equiv W_{8r,r}(x)\equiv \int_{\B {8r} {x}\setminus \B {r}{x}} \frac{\abs{(y-x)\cdot \nabla u(y)}^2}{\abs {y-x}^m}d\Vol(y) \geq 0\, .
\end{align}
Note that for an approximate harmonic map, by \eqref{eq_hat'} we have the bound
\begin{gather}
 W_r(x)\leq \hat \theta(x,8r)-\hat\theta(x,r)\, .
\end{gather}

Note that, at least philosophically, bounds on $W_r(x)$ and $W_r(y)$, for $\abs {x-y}\leq r$, give bounds on the scale-invariant $L^2$ norm of $(x-y)\cdot \nabla u$ in an annulus around $x$ and $y$. This is an easy consequence of the fact that for all $z\in \R^m$, the vectors $(z-y)$ and $(z-x)$ always span the vector $(x-y)$. In this section, we will exploit this simple idea and some easy tricks to prove $\beta_2$ estimates in a very general setting.

\vspace{3mm}

The main estimate in this section is the following. Note that, up to minor technical details, this theorem is similar to \cite[theorem 7.1]{nava+}.
\begin{theorem}\label{th_best_app}
Let $u$ be as above, and fix $\epsilon>0$, $0<r\leq 1$ and $x\in \B 1 0$. There exists a constant $C_1(m,\KN,\Lambda,\gamma,\epsilon)>0$ such that if $\cF\leq \delta_5$ and $\B {8r} x$ is $(0,\delta_5)$-symmetric but not $(k+1,\epsilon)$-symmetric, then for any nonnegative finite measure $\mu$ on $\B {r} x$ we can estimate
\begin{align}\label{eq_D2W}
 \beta^k_{2,\mu}(x,r)^2  = \inf_{V} r^{-2-k}\int_{\B {r}{x}} d^2(x,V)\,d\mu(x) \leq C_1 r^{-k}\int_{\B {r}{x}} W_r (x)\,d\mu(x)\, , 
\end{align}
where the $\inf$ is taken over all $k$-dimensional affine subspaces $V\subseteq \R^m$.
\end{theorem}

\begin{remark}
 As it will be clear in the proof, $\delta_5$ here is the same as the one given by Lemma \ref{lemma_lower_bound}. Moreover, $C_1=c(m) \delta_5^{-1}$.
\end{remark}

\begin{remark}
 Note that the quantity on the rhs of this theorem can be easily estimated in terms of $\hat \theta$ by \eqref{eq_hat'}.
\end{remark}

\vspace{.5 cm}

The proof of this theorem hinges on some manipulations over the eigenvalues and eigenvectors of the ``inertia matrix'' associated to every measure $\mu$.
\subsubsection{Eigenvalue and eigenvectors of the matrix associated to a measure}\label{sec_Q}

Let us consider a probability measure $\mu$ with support in $\B 1 0$, and let $x_{cm}$ be its center of mass, i.e.:
\begin{gather}\label{eq_cm}
x_{cm}=x_{cm}(\mu)\equiv \int x \, d\mu(x)\, .
\end{gather}
Consider the bilinear quadratic form $Q(v,w)$ defined by
 \begin{gather}
  Q(v,w)\equiv \int \qua{(x-x_{cm})\cdot v }\qua{(x-x_{cm})\cdot w}\, d\mu(x) \, .
 \end{gather}
 In this section, we study the eigenvalue and eigenvectors of $Q$ and their relations with the $\beta_2$ defined above.

\begin{definition}\label{deph_lambda_v}
Given a probability measure $\mu\in \B 1 0$, we set $\lambda_1(\mu),\cdots,\lambda_m(\mu)$ to be the eigenvalues of $Q(\mu)$ in decreasing order, and $v_1(\mu),\cdots,v_m(\mu)$ to be its eigenvectors. In case one eigenvalue has higher multiplicity, we take any choice of orthonormal eigenvectors inside the eigenspace.
\end{definition}

Note that by definition of eigenvectors, we have
\begin{gather}\label{eq_Qv}
 Q(v_k)=\lambda_k v_k = \int \qua{(x-x_{cm})\cdot v_k }(x-x_{cm})\, d\mu(x)
\end{gather}
We also have a variational characterization of the eigenvalues given by
\begin{gather}
\lambda_1=\lambda_1(\mu)\equiv \max_{\abs v ^2=1} \int \abs{(x-x_{cm})\cdot v}^2\, d\mu(x)\, .
\end{gather}
and $v_1=v_1(\mu)$ is any of the norm $1$ vectors obtaining this maximum. By induction, we also have
\begin{gather}
\lambda_{k+1}=\lambda_{k+1}(\mu)\equiv \max\cur{ \int \abs{(x-x_{cm})\cdot v}^2\, d\mu(x) \ \ s.t. \ \ \abs v ^2=1\,  \text{and} \ \ \forall i\leq k\, , \  v\cdot v_i=0}\, , 
\end{gather}
and $v_{k+1}$ is a vector obtaining this maximum. Note that, by definition of $v_k$, the subspace $V_k=x_{cm}+\operatorname{span}\{v_1,\ldots,v_k\}$ is the $k$-dimensional affine subspace (or one of the subspaces) achieving the minimum in the $\beta_2$. In other words
\begin{gather}\label{eq_asdf}
\min_{V\subseteq \R^m\, , \ \ \operatorname{dim}(V)=k} \int d^2(x,V)\,d\mu(x) = \int d^2(x,V_k)\,d\mu(x) = \lambda_{k+1}(\mu)+\cdots+\lambda_m(\mu)\, .
\end{gather}
\begin{remark}
 Note that evidently $V_k$ must pass through the center of mass of $\mu$. This is an immediate corollary of Jensen's inequality (or Steiner's theorem).
\end{remark}

\vspace{.5cm}

By simple manipulations with $\lambda_k$ and $v_k$, we obtain the following important estimate:
\begin{proposition}\label{prop_best_V}
Let $u:\B 9 0\to N$ be an $H^1$ map, and let $\mu$ be a probability measure on $\B 1 0$ with $\lambda_k(\mu),v_k(\mu)$ defined as in Definition \ref{deph_lambda_v}.  Then there exists $C(m)>0$ such that
\begin{align}
\lambda_k \int_{A_{3,4}(0)} \abs{v_k\cdot \nabla u(z)}^2\,d\Vol (z) \leq C(m) \int W_0(x)\, d\mu(x)\, .
\end{align}
\end{proposition}
\begin{proof}
For simplicity, and without essential loss of generality, we assume that $x_{cm}(\mu)= 0$ (otherwise a simple translation will do the trick). 

For any $z\in A_{3,4}$ and $k=1,\cdots,m$, we take the scalar product of \eqref{eq_Qv} with $\nabla u(z)$, and obtain
\begin{align}\label{eq_aaa}
\lambda_k\ton{ v_k\cdot \nabla u(z)}&= \int \ton{x\cdot v_k}\ton{\nabla u(z)\cdot x}\, d\mu(x) \, ,
\end{align}
By definition of center of mass (see \eqref{eq_cm}), we have for all fixed $z$:
\begin{gather}
\int x\cdot z \ d\mu(x) = x_{cm}\cdot z= 0\, . 
\end{gather}
Thus we can re-write \eqref{eq_aaa} in the form:
\begin{gather}
\lambda_k \ton{\nabla u(z)\cdot v_k}= \int \ton{x\cdot v_k}  \qua{\nabla u(z)\cdot(x-z)}\,d\mu(x) \, .
\end{gather}
A simple application of \hol inequality tells us that for all fixed $z$:
\begin{gather}
\lambda_k^2\abs{\nabla u(z) \cdot v_k}^2 \leq \lambda_k\int \abs{\nabla u(z)\cdot (x-z)}^2\,d\mu(x) \, .
\end{gather}
Note that we can evidently assume $\lambda_k>0$, otherwise there's nothing to prove. By integrating both sides of the previous inequality on $A_{3,4}(0)$, we get
\begin{gather}\label{e:best_subspace_estimate:1}
\lambda_k \int_{A_{3,4}(0)}\abs{\nabla u(z)\cdot v_k}^2\,d\Vol(z)\leq  \int\int_{A_{3,4}(0)} \abs{\nabla u(z) \cdot (x-z)}^2\,d\Vol (z)\,d\mu(x) \leq\\
\leq \int\int_{A_{3,4}(0)} \frac{\abs{\nabla u(z) \cdot (x-z)}^2}{\abs{x-z}^m} \abs{x-z}^m\,d\Vol (z)\,d\mu(x)\leq\\
\leq C(m) \int\int_{A_{1,8}(x)} \frac{\abs{\nabla u(z) \cdot (x-z)}^2}{\abs{x-z}^m} \,d\Vol (z)\,d\mu(x)\leq C(m) \int W_0(x)\ d\mu(x)\, ,
\end{gather}
as desired.
\end{proof}

\subsubsection{Proof of Theorem \ref{th_best_app}}

We are now in a position to prove Theorem \ref{th_best_app}. By rescaling \eqref{eq_D2W}, we can assume for convenience that $\mu(\B 1 0)=1$. Since we have ordered $\lambda_k$ to be decreasing in value, and by \eqref{eq_asdf}, we have
\begin{gather}\label{eq_thbest_aaa}
\beta^k_{2,\mu}(0,1)^2  = \lambda_{k+1}(\mu)+\cdots+\lambda_m(\mu)\leq (m-k)\lambda_{k+1}\, .
\end{gather}
By applying Proposition \ref{prop_best_V} to each $j=1,\cdots,k+1$, we obtain
\begin{gather}
\sum_{j=1}^{k+1} \lambda_j \int_{A_{3,4}(0)} \abs{\nabla u(z)\cdot v_j}^2\,d\Vol (z) \leq (k+1) C \int W_0(x)\, d\mu(x)\, .
\end{gather}
Let $V^{k+1}=\operatorname{span}\ton{v_1,\cdots,v_{k+1}}$ be the linear part of the best $k+1$-dimensional subspace of $\mu$. Given that $\lambda_j$ are decreasing in $j$ (by definition), the last estimate leads to
\begin{gather}
\lambda_{k+1}\int_{A_{3,4}(0)} \abs{V^{k+1}\cdot \nabla u(z)}^2\,d\Vol (z)= \lambda_{k+1} \sum_{j=1}^{k+1}\int_{A_{3,4}(p)} \abs{\nabla u(z)\cdot v_j}^2\,d\Vol (z)\leq C \int W_0(x)\, d\mu(x)\, .
\end{gather}

By assumption, we know that $\B 8 0$ is $(0,\delta_5)$-symmetric and not $(k+1,\epsilon)$-symmetric. Thus, by Lemma \ref{lemma_lower_bound}, 
\begin{gather}
\int_{A_{3,4}(p)} \abs{\nabla u(z)\cdot V^{k+1}}^2\,d\Vol (z)\geq 8^{m-2}\delta_5\, .
\end{gather}
This allows us to estimate 
\begin{gather}
c(m)\delta_5\lambda_{k+1}\leq \lambda_{k+1} \int_{A_{3,4}(0)} \abs{\nabla u(z)\cdot V^{k+1} }^2\,d\Vol (z)\leq C \int W_0(x)\, d\mu(x)\, .
\end{gather}
Since $\delta_5$ is a positive constant depending only on $(m,\KN,\Lambda,\gamma,\epsilon)$, and by \eqref{eq_thbest_aaa}, we can conclude
\begin{gather}
 \beta^k_{2,\mu}(0,1)^2 \leq C(m,\KN,\Lambda,\gamma,\epsilon) \int W_0(x)\, d\mu(x)
\end{gather}
as desired. 

\begin{flushright}
 $\square$
\end{flushright}

\subsection{Covering argument}\label{sec_cov}
In this subsection, we prove the inductive covering argument needed for the main theorem. We split this covering argument into two lemmas: in the first one, we keep refining inductively a covering by balls until all but a controlled amount of points in our balls have some definite drop in $\hat \theta$, and in the second one we show that this controlled amount of points without drop is small so that they can be ``ignored''.

\begin{lemma}[Covering Lemma I]\label{lemma_cover}
Let $u:\B 3 0\to N$ be an approximately harmonic map satisfying \eqref{eq_app1} and \eqref{eq_app_stat}, with the conditions \eqref{eq_f}.  Fix any $\epsilon>0$, $0<\rho<\rho(m)\leq 100^{-1}$, and $0<r<R, \ 0<R\leq 1$ arbitrary, set $E = \sup_{x\in B_{2R}(0)\cap \cS} \hat \theta_1(x)$, and assume the uniform bound $E\leq \Lambda$. There exists $\delta=\delta(m,\KN,\Lambda,\gamma,\rho,\epsilon)>0$ and $C_V(m)$ such that the following is true.

If $\cF<\delta$ then for any subset $\cS \subseteq \cS^k_{\epsilon,\delta r }$ there exists a finite covering of $\cS\cap \B R 0 $ such that 
\begin{gather}
 \cS\cap \B R 0 \subseteq \bigcup_{x\in \cC} \B {r_x}{x} \quad \text{ with } \ \ r_x\geq r \ \ \text{ and}\ \ \sum_{x\in \cC} r_x^k\leq C_V(m)R^k\, .
\end{gather}
Moreover, for each $x\in \cC$, one of the following is verified
\begin{enumerate}[i)]
 \item\label{it_a}  $r_x=r$
 \item\label{it_c}  the set of points $F_x\equiv\cur{y\in \cS\cap \B {2r_x}{x} \ \ s.t. \ \ \theta(y,\rho r_x/10)>E-\delta }$
is contained in $\B {\rho r_x/5}{L_x}\cap \B {2r_x}{x}$, where $L_x$ is some $k-1$ dimensional affine subspace.
\end{enumerate}
\end{lemma}

\begin{remark}
 By the scale-invariance properties of $\hat \theta$, it is clear that for simplicity we can assume wlog that $R=1$.
\end{remark}

\begin{remark}
 Note that the set $F_x$ may be empty. 
\end{remark}


\begin{remark}
 For convenience, and without any loss of generality, we will assume in the proof that $r$ is some (positive) power of $\rho$, and that $\rho$ is some (negative) power of $2$. In particular:
 \begin{gather}\label{eq_convenience}
  r=\rho^{\bar j}\, \quad \text{ and } \quad \rho=2^{-a}\, , \quad \text{ with } \ \ a,\bar j \in \N\, .
 \end{gather}

\end{remark}

\subsubsection{Proof of Lemma \ref{lemma_cover}}
 The idea of the proof is the following. We are going to build inductively on $i$ a covering of the set $\cS^{k}_{\epsilon,r}$ by a family of balls of radius $r_i=\rho^{i}$. In the inductive step, we will look at each ball of radius $r_i$ and determine if this is a ``good'' or a ``bad'' ball according to how many points inside this ball have $\hat \theta(y,\rho r_i/10)\geq E-\delta$. 
 
 If this set of points ``effectively span'' some $k$-dimensional affine subspace $V$, then we will apply Lemma \ref{prop_k+1_pinch} in order to see that the whole set $\cS^k_{\epsilon,r}\cap \B {r_i}{x}$ is contained in small neighborhood of $V$. Moreover, using Lemma \ref{lemma_unipinch}, we will see that we can cover the whole neighborhood of $V$ by balls with uniform radius, and this covering will satisfy the assumptions of the discrete Reifenberg theorem. These balls are the good balls.
 
 If this set of points is empty, or it does not ``effectively span'' something $k$-dimensional, then we will stop refining our covering, because by definition condition \eqref{it_c} is verified.

 The uniform $k$-dimensional content estimates will follow from the discrete Reifenberg theorem \ref{th_disc_reif} applied to the natural measure associated with this covering. The $\beta_2$ estimates needed to apply the Reifenberg theorem are a consequence of Section \ref{sec_dist_est}.

 \subsubsection{Inductive covering: first step} Consider the map $u:\B 3 0\to N$, let $\cS\subseteq \cS^k_{\epsilon,\delta r}$ be an arbitrary subset and define the set
 \begin{gather}
  F=\cur{y\in \B 2 0 \cap \cS \ \ s.t. \ \ \hat \theta(y,\rho/10)> E-\delta}\, .
 \end{gather}
 If there exists a $k-1$-dimensional subspace $L$ such that $F\subset \B {\rho/5}{L}$, then there's nothing to prove. In this case, we call $\B 1 0$ a \textit{bad} ball.

 Otherwise, we say that $\B 1 0$ is a \textit{good} ball. In this second case, let $V$ be a $k$-dimensional subspace which is $(\rho/10)$-effectively spanned by the set $F$. Thus by definition there exists $\cur{y_j}_{j=0}^k\subset F$ that $(\rho/10)$-effectively span $V$. For $\delta$ sufficiently small, we can apply Lemma \ref{prop_k+1_pinch} to $\B 1 0$, we obtain that
 \begin{gather}
  \cS^k_{\epsilon,\delta r}\cap \B 1 0  \subset \B {\rho/5}{V}\, .
 \end{gather}
Consider a finite covering of $\B {\rho/5}{V}\cap B_1$ by balls $\cur{\B{\rho}{x}}_{x\in \cC}$ such that
\begin{enumerate}
 \item $x\in V\cap \B 1 0$
 \item if $x\neq y$, then $\B {\rho/5}{x}\cap \B {\rho/5}{y}=\emptyset$
\end{enumerate}
Note that, by Lemma \ref{lemma_unipinch}, we have for all $x\in \cC$:
\begin{gather}
 \hat \theta(x,\rho/10)\geq E-\eta\, ,
\end{gather}
as long as $\delta$ is sufficiently small. Under the same smallness assumption, Lemma \ref{lemma_ksym_pinch} implies that for all $x$ we have $x\in \cS^k_{\rho,\epsilon/2}$. We will need these two properties later on in order to apply the discrete Reifenberg Theorem \ref{th_disc_reif} to the measure associated to our final covering.

This completes the base step of the inductive covering we will be constructing in the next subsection. Now we will consider any of the balls $\B {\rho}{x}$ in this covering and start over the process.

\subsubsection{Inductive step}

We will build by induction a sequence of coverings 
\begin{align}
\cS\subseteq \bigcup_{x\in \cC^j} B_{r^j_x}(x)= \bigcup_{x\in \cC^j_b} B_{r^j_x}(x)\cup \bigcup_{x\in\cC^j_g} B_{r^j_x}(x) \equiv \B{r^j_x}{\cC^j_b}\cup \B{r^j_x}{\cC^j_g}	\, ,
\end{align}
where $\cC^j_b$ will represent the centers of a collection of ``bad balls'' and $\cC^j_g$ will represent the centers of a collection of ``good balls'' such that
\begin{enumerate}
 \item If $x\in \cC_b^j$ then $r^j_x\geq \rho^j$ and the set $F_x= \cur{y\in \cS \cap \B {2r_x^j}{x} \ \ s.t. \ \ \hat \theta(y,\rho r^j_x/10)\geq E-\delta}$ is contained in some $\B {\rho r_x^j/5}{L_x}$, where $L_x$ is a $k-1$-dimensional affine subspace.
 \item If $x\in \cC_g^j$ then $r^j_x\equiv \rho^j$ and the set $F_x=\cur{y\in \cS \cap \B {2r^j_x}{x} \ \ s.t. \ \ \hat \theta(y,\rho r^j_x/10)\geq E-\delta}$ $(\rho r_x^j/10)$-effectively spans a $k$-dimensional affine subspace $V_x$.
 \item For all $x\neq y\in \cC^j$ we have $\B {r_x/5}{x} \cap \B {r_y/5}{y}=\emptyset$.
 \item For all $x\in \cC^j$ we have $\hat \theta(x,r_x)\geq E-\eta$.
 \item For all $x\in \cC^j$ and for all $s\in [r_x,1]$, $\B s {x}$ is not $(k+1,\epsilon/2)$-symmetric.
\end{enumerate}
Suppose that we have this covering for some $j$, and consider the set
\begin{gather}\label{eq_cov1}
 R_j=\cS \setminus \bigcup_{x\in \cC_b^j}\B {r_x}{x} = \cS\setminus  \B {r_x}{\cC_b^j}\, .
\end{gather}
Note that by definition this set is contained in $ \B {\rho^j}{\cC_g^j}$.  For each $x\in \cC_g^j$, we know that $F_x$ $[\rho^{j+1}/10]$-effectively spans a $k$-dimensional subspace $V_x$. As seen in the first inductive step, by Proposition \ref{prop_k+1_pinch} we have that
 \begin{gather}\label{eq_cov2}
  \cS^k_{\epsilon,\delta r}\cap \B{2\rho^j}{x} \subset \B {\rho^{j+1}/5}{V_x}\, 
 \end{gather}
for all $x\in \cC_g^j$ as long as
\begin{gather}\label{eq_da}
 \delta\leq \delta_3 (m,\Lambda,\KN,\gamma,\rho,\epsilon)\, .
\end{gather}
In order to build an open covering of $R_j$, consider the set
\begin{gather}
 A=\bigcup_{x \in \cC^j_g} \ton{\B {\rho^j}{x} \cap V_x} \setminus \B {r_x/2}{\cC_b^j}\, .
\end{gather}
By \eqref{eq_cov1} and \eqref{eq_cov2}, and since $\rho < 100^{-1}$, we have 
\begin{gather}
 R_j \subseteq \B{\rho^{j+1}/5 }{A}\, .
\end{gather}
Now first note that by the definition of $A$ and since $\rho \leq 100^{-1}$, all of these balls are disjoint from $\B {r_x/10}{\cC^j_b}$. Moreover, by Lemma \ref{lemma_unipinch}, if we choose $\delta$ sufficiently small, for all $y\in A$ we have
\begin{gather}\label{eq_db}
 \hat \theta\ton{y,\rho^{j+1}/10}\geq E-\eta\, .
\end{gather}
In particular, we need
\begin{gather}
 0<\delta\leq\delta_6(m,\Lambda,\KN,\rho,\gamma,\eta)\, .
\end{gather}
Furthermore, if we choose $\delta$ small enough, by Lemma \ref{lemma_ksym_pinch}, we obtain that for all $s\in [\rho,1]$ and for all $y\in A$
\begin{gather}\label{eq_notk+1}
\B s y \ \ \text{is not}\ \ (k+1,\epsilon/2)\text{-symmetric}\, .
 \end{gather}
In particular, we need
\begin{gather}\label{eq_eta_last}
 0<\eta\leq \delta_7(m,\Lambda,\KN,\rho,\gamma,\epsilon)\,  \quad \Longleftarrow \quad 0<\delta \leq \min\cur{\delta_7(m,\Lambda,\KN,\rho,\gamma,\epsilon), \delta_6(m,\Lambda,\KN,\rho,\gamma,\delta_7)}\, .
\end{gather}

Now consider a (finite) Vitali subcovering of this set given by
\begin{gather}\label{eq_cov3}
 R_j \subseteq \bigcup_{x\in \cC^A} \B{\rho^{j+1} }{x}\, .
\end{gather}
We can classify all the balls in this covering into good and bad according to how spread their set $F$ is. In particular, for all $x\in\cC^A$ consider as above the set
\begin{gather}\label{eq_cov4}
 F_x= \cur{y\in \cS \cap \B {2\rho^{j+1} }{x} \ \ s.t. \ \ \hat \theta\ton{y,\rho^{j+2}/10 }\geq E-\delta}\, .
\end{gather}
If $F_x$ $[\rho^{j+2}/10]$-effectively spans a $k$ dimensional subspace $V_x$, then we say that $\B{\rho^{j+1} }{x_t}$ is a good ball, and we put $x\in \cC^A_g$. Otherwise, we say that $\B{\rho^{j+1} }{x}$ is a bad ball, and we put $x\in \cC^A_b$.

We define
\begin{gather}
 \cC_b^{j+1}=\cC_b^j \cup \cC^A_b \, , \quad \cC_g^{j+1}= \cC^A_g\, .
\end{gather}
Note that the set of bad balls contains \textit{all} the bad balls encountered at any previous step. On the contrary, good balls get refined at each stage, and at each induction step the previous bad balls disappear from the set $\cC_g$.

Now the induction is complete. Indeed, property $1$ and $2$ are a direct consequence of the definition of $\cC_g$ and $\cC_b$.  Property $3$ comes from the definition of $A$ and the Vitali covering lemma. Finally, property $4$ is a consequence of \eqref{eq_db} and property $5$ comes from \eqref{eq_notk+1}. 

\subsubsection{Volume estimates}\label{sec_volest}
Now we are in a position to prove the desired volume estimates, and in particular
\begin{gather}\label{eq_obj}
 \sum_{x\in \cC} r_x^k\leq C_V(m)\, ,
\end{gather}
where $\cC=\cC^{\bar j}$ for $\bar j$ such that $\rho^{\bar j}=r$.

We will prove this estimate by an induction on the radius. For convenience, we define the measure 
\begin{gather}\label{eq_deph_mu}
 \mu=\omega_k \sum_{x\in \cC} r_x^k \delta_{x} \, .
\end{gather}

\paragraph{Upwards induction}
For all $t\in (0,1]$, set $\cC_t = \cur{x \in \cC \ \ s.t. \ \ r_x\leq t}$, and define the measure
\begin{gather}
 \mu_t \equiv \omega_k \sum_{x\in \cC_t} r_x^k \delta_{x}  \leq \mu\, .
\end{gather}
Now we want to prove inductively on $t=r,2r,2^2r,2^3r,\cdots,1/8$ that for some universal constant $C_R(m)$, for all $x\in \B 3 0$ and $s\geq r$ we have
\begin{gather}\label{eq_up_ind}
 \mu_t(\B t x) \equiv \ton{\sum_{x\in \cC \ s.t. \ r_x\leq t} \omega_k r_x^k \delta_{x} }(\B t x) \leq C_R(m) t^k\, .
\end{gather}
Note that $C_R(m)$ is the constant in Theorem \ref{th_disc_reif}. Note also that $\mu_1=\mu$, so at the last step of the induction we will have recovered an estimate for the whole $\mu$, up to a covering of $\B 1 0$ by balls $\B{1/8}{p_i}$. In other words, we prove \eqref{eq_obj} with 
\begin{gather}\label{eq_CV}
C_V(m)=c(m) C_R(m)\, .
\end{gather}

Note that the base step is easily seen to be true for $t=r$. Indeed, at this stage we have
\begin{gather}
 \mu_r = \sum_{x\in \cC_r } \omega_k r^k \delta_{x}\, ,
\end{gather}
where all $\B {r/5} {x_i}$ are disjoint. Thus we immediately have $\mu_r(\B r x)\leq c(m)r^k$.

\vspace{3mm}

Now, suppose that we have proven \eqref{eq_up_ind} for $t\leq 2^j r$, we will show that \eqref{eq_up_ind} holds also for $t=2^{j+1}r$.

\paragraph{Rough estimate} First of all, we note that by a very bad estimate we have for all $x\in \B 1 0$:
\begin{gather}\label{eq_rough}
 \mu_{2\bar r}(\B {2\bar r} {x}) \leq c(m) C_R(m) (2\bar r)^k\, ,
\end{gather}
where for convenience we have set $\bar r=2^j r$.  Indeed, we can split $\mu_{2\bar r}$ into
\begin{gather}
\mu_{2\bar r}= \mu_{\bar r} + \tilde \mu_{2\bar r}\equiv \sum_{x\in \cC_{\bar r}} \omega_k r_x^k\delta_x + \sum_{x\in \cC \ s.t. \ r_x\in (\bar r,2\bar r]} \omega_k r_x^k\delta_x \, .
\end{gather}
Take a covering of $\B {2\bar r}{x}$ by balls $\B {\bar r}{y_i}$ such that $\B {\bar r/2}{y_i}$ are disjoint. The number of these balls has a universal bound $c(m)$, and by induction we have
\begin{gather}
 \mu_{\bar r} (\B {2\bar r}{x})\leq \sum_i \mu_{\bar r} (\B {\bar r}{y_i}) \leq c(m) C_R(m) \bar r^k\, .
\end{gather}
As for the other part of $\mu$, by definition of this measure all the balls $\B {r_x/5}{x}$ are pairwise disjoint, and so we get immediately
\begin{gather}
 \tilde \mu_{2\bar r}(\B {2\bar r}{x})\leq c(m) (2\bar r)^k\, .
\end{gather}

\paragraph{Reifenberg estimates}
We will show inductively that we can apply Theorem \ref{th_disc_reif} to the measures $\mu_{2\bar r}$ on each fixed $\B {2\bar r}{x}$. For convenience, we set \begin{gather}
\bar \mu = \mu_{2\bar r}|_{\B {2\bar r}{x}}\, .                                                                                                                  \end{gather}

Note that for all $x\in \supp \mu$, and all $s\in [r_x,1]$, we have $\hat \theta(x,s)-\hat \theta(x,s/2)<\eta$ because $\hat \theta(x_i,s)\leq E$ by monotonicity of $\hat \theta$ and by definition of $E$, and $\hat \theta(x_i,s/2)\geq E-\eta$ by condition (4) of our constructed covering.  Now we can choose $\eta$ small enough so that for all $x\in \supp \mu$ and $0<s\leq 1$ we have the $\beta_2$ estimate
\begin{gather}\label{eq_vol1}
 \beta_{2,\bar\mu}(x,s)^2  \stackrel{thm. \ref{th_best_app}}{\leq} C_1 s^{-k}\int_{\B {s}{x}} \hat W_s (y)\,d\bar \mu(y)\, ,
\end{gather}
where we have set for all $x\in \supp \mu$:
\begin{gather}
 \hat W_s(x) = \begin{cases}
                W_s(x_i) & \text{ if } s> r_x\, ,\\
                0 & \text{ if } s\leq r_x\, .
               \end{cases}
\end{gather}

Indeed, for $s\leq r_x$, $\supp \mu \cap \B {s}{x}=\cur{x}$, and there's nothing to prove. If $s\geq r_x$, then for all $y\in \B {r_x}{x}$, $r_y< s$ by construction of $\mu$.

Now for $r_x\leq s \leq 1/8$, the ball $\B {8s} {x}$ is not $(k+1,\epsilon/2)$-symmetric by \eqref{eq_notk+1}. Let $\delta_5(m,\KN,\Lambda,\gamma,\epsilon)$ be the parameter found in Lemma \ref{lemma_lower_bound} and Theorem \ref{th_best_app}. By Proposition \ref{prop_1pinch}, we can choose a threshold
\begin{gather}\label{eq_eta0}
 \eta_0(m,\KN,\Lambda,\gamma,\epsilon)= \delta_1(m,\KN,\Lambda,\gamma,\epsilon,\delta_5)>0
\end{gather}
such that $\hat \theta(x,8s)-\hat\theta(x,4s)<\eta$ with $\eta\leq \eta_0$ implies that $\B {8s} {x_i}$ is $(0,\delta_5)$-symmetric. Thus all the assumptions of Theorem \ref{th_best_app} are satisfied, and we have the estimate \eqref{eq_vol1} as desired.

Now we can prove that for all $y\in \B {2\bar r}{x}$, and $r\leq 2\bar r$, we have
\begin{gather}\label{eq_vol_reif}
 \int_{B_r(y)}\ton{\int_0^r \beta^k_{2,\bar\mu}(z,s)^2 \,{\frac{ds}{s}}}\, d\bar \mu(z)<c(m) C_1 C_R^2 \eta r^{k}\, .
\end{gather}
Indeed, by \eqref{eq_vol1} we can estimate for all $s\leq r$:
\begin{gather}
 \int_{B_r(y)} \beta^k_{2,\bar\mu}(z,s)^2 \, d\bar \mu(z)\leq C_1 s^{-k}\int_{\B r y} \qua{\int_{\B {s}{z}} \hat W_s (t)\, d\bar \mu(t)} d\bar \mu (z)\, .
\end{gather}
Now, on $\B s z$, either $\bar \mu = \mu_s|_{\B {2\bar r}{x}} $, or there exists an $x\in \supp \mu \cap \B s z$ with $r_x>s$. Since $z\in \supp \mu$ as well, by construction we have $z=x=\supp \mu \cap \B s z$, and $\hat W_s(z)=0$. Thus in either case we have
\begin{gather}
 \int_{B_r(y)} \beta^k_{2,\bar\mu}(z,s)^2\, d\bar \mu(z)\leq C_1 s^{-k}\int_{\B r y\cap \B {2\bar r}{x}} \qua{\int_{\B {s}{z}\cap\B {2\bar r}{x}} \hat W_s (t)\, d\mu_s(t)} d\mu_s (z)\, .
\end{gather}
By induction, and by the rough estimates in \eqref{eq_rough}, for all $s\in (0,2 \bar r]$ and $z\in \B 1 0$ we can estimate
\begin{gather}
 \mu_s(\B s z)\leq c(m) C_R s^k\, .
\end{gather}
Thus we obtain
\begin{gather}
 \int_{B_r(y)} \beta^k_{2,\bar\mu}(z,s)^2\, d\bar \mu(z)\leq c(m)C_1C_R \int_{\B {r+s} y \cap \B {2\bar r}{x}} \hat W_s (z)d\mu_s (z)=c(m)C_1C_R \int_{\B {r+s} y} \hat W_s (z)d\bar \mu (z)\, .
\end{gather}
This yields
\begin{gather}
  \int_{\B r y}\ton{\int_0^r \beta^k_{2,\bar\mu}(z,s)^2\,{\frac{ds}{s}}}\, d\bar \mu(z)
  \leq c(m) C_1 C_R \int_{\B {2r} y} \qua{\int_0^r \hat W_s (z) \frac{ds}{s}}d\bar \mu (z)\, .
\end{gather}
Note that for all $x \in \supp \mu$ and $r\leq 2\bar r\leq 1/8$, we have 
\begin{gather}
 \int_0^r \hat W_s (x) \frac{ds}{s}=\int_{r_x}^r \hat W_s (x) \frac{ds}{s}\leq \int_{r_x}^{1/8} \hat W_s (x) \frac{ds}{s}\stackrel{\eqref{eq_hat'}}{\leq} c \qua{\hat \theta(x,1)-\hat \theta(x,r_x)}\leq c\eta\, .
\end{gather}
Thus, using again the induction hypothesis and the rough estimates \eqref{eq_rough}, we prove \eqref{eq_vol_reif}.

If we choose $\eta$ small enough, in particular
\begin{gather}\label{eq_eta1}
 \eta \leq \eta_1(m,\KN,\Lambda,\gamma,\epsilon)=c(m)\frac{\delta_R^2}{C_1C_R^2}\, ,
\end{gather}
we can apply Theorem \ref{th_disc_reif} to $\bar \mu$ and obtain \eqref{eq_up_ind} as wanted.

\vspace{2mm}

The only thing left to do is to choose $\delta=\delta(m,\KN,\Lambda,\gamma,\rho,\epsilon)>0$ in such a way that \eqref{eq_db} is satisfied with
\begin{gather}\label{eq_eta}
 \eta \leq \min\cur{\eta_0,\eta_1,\delta_7}\, 
\end{gather}
and also \eqref{eq_da} is satisfied. Given \eqref{eq_eta0} and \eqref{eq_eta1}, as noted above this is a simple application of Lemmas \ref{lemma_unipinch} and \ref{lemma_ksym_pinch}.  This finishes the proof of Lemma \ref{lemma_cover}.

\begin{flushright}
 $\square$
\end{flushright}

\subsubsection{Second covering lemma}
By repeating this covering argument over bad balls, we obtain the following
\begin{lemma}[Covering Lemma II]\label{lemma_coverII}
Let $u:\B 3 0\to N$ be an approximately harmonic map satisfying \eqref{eq_app1} and \eqref{eq_app_stat}, with the conditions \eqref{eq_f}.  Fix any $\epsilon>0$ and $0<r\leq R\, , \ 0<R\leq 1$, set $E = \sup_{x\in \B {2R} 0 \cap \cS} \hat \theta_1(x)$, and assume the uniform bound $E\leq \Lambda$. There exists $\delta=\delta(m,\KN,\Lambda,\gamma,\epsilon)>0$ and $C_F(m)$ such that the following is true.

If $\cF<\delta$, for any subset $\cS \subseteq \cS^k_{\epsilon,\delta r}$, there exists a finite covering of $\cS\cap \B R 0 $ such that 
\begin{gather}
 \cS\cap \B R 0 \subseteq \bigcup_{x\in \cC} \B {r_x}{x}\, , \quad \text{ with } \ \ r_x\geq r \ \ \text{ and}\ \ \sum_{x\in \cC} r_x^k\leq C_F(m) R^k\, .
\end{gather}
Moreover, for each $x\in \cC$,
\begin{enumerate}[i)]
 \item \label{it_rec_1} either $r_x= r$
 \item \label{it_rec_2} or we have the following uniform energy drop
  \begin{gather}\label{eq_Edrop}
   \forall y\in B_{r_x}(x) \cap \cS\, , \ \ \hat \theta(y,r_x/10) \leq E- \delta\, .
  \end{gather}
\end{enumerate}
\end{lemma}

\begin{remark}
 As for the previous covering lemma, also in this case we can assume for simplicity and wlog that $R=1$. 
\end{remark}

\begin{proof}
 We need to refine the covering of the previous lemma. Recall that by lemma \ref{lemma_cover} we have a covering of $\cS\cap \B 1 0$ given by
\begin{gather}\label{eq_covering_start}
 \cS\cap \B 1 0 \subseteq \bigcup_{x\in \cC} \B {r}{x}\equiv \bigcup_{x\in \cC_r} \B {r}{x} \cup \bigcup_{x\in \cC_+} \B {r_x} {x} \quad \text{ with } \ \ r_x\geq r \ \ \text{ and}\ \ \sum_{x\in \cC_r\cup \cC_+} r_x^k\leq C_V(m)\, ,
\end{gather}
where we have set
\begin{gather}
\cC_r=\cur{x\in \cC \ s.t. \ \ r_x=r}\quad  \text{and}\quad \cC_+=\cur{x\in \cC \ \ s.t. \ \ r_x>r}\, , \quad \cC=\cC_r\cup \cC_+\, . 
\end{gather}
We will of course keep $\cC_r$ as part of our final covering, while we will refine the covering on each of the balls $\cur{\B {r_x}{x}}_{x\in \cC_+}$ in an inductive way. By item \eqref{it_c} of lemma \ref{lemma_cover}, for each $x\in \cC_+$ the set $F_x\equiv\cur{y\in \cS\cap \B {2r_x}{x} \ \ s.t. \ \ \theta(y,\rho r_i/10)>E-\delta }$ is close to a $k-1$-dimensional space.
Assuming that $F_x=\emptyset$, all we need to do in order to achieve \eqref{eq_Edrop} is to re-cover $\B {r_x}{x}$ with balls $\cur{\B {\rho r_x}{y}}_{y\in \cC^{(1,f)}_x}$. These balls are the final covering we are looking for. Evidently, the number of these balls is bounded by a constant $C_f(m,\rho)$.

If $F_x\neq \emptyset$, we need to exploit the fact that we still know $F_x\subseteq \B {\rho r_x/5}{L_x}\cap \B {2r_x}{x}$, where $L_x$ is at most $k-1$ dimensional. Thus we can cover $\B {r_x}{x}\setminus \B {\rho r_x}{F_x}$ as above, and cover $\B {\rho r_x}{F_x}$ separately by balls $\cur{\B {\rho r_x}{y}}_{y\in \cC^{(1,b)}_{x}}$.
On these ``bad balls'', we will not be able to obtain any information on the energy drop over these new balls in the covering. However, their $k$-dimensional content is small since $F_x$ behaves like a $k-1$ dimensional set.
This will allow us to start over on each of these bad balls separately, and keep a uniform $k$-dimensional estimate on the content of the final covering.  More precisely:

\subsubsection{Re-covering of bad balls: Induction}
In detail, we will build by induction on $i$ a sequence of coverings of $\cS\subseteq \cS^k_{\epsilon,\delta r}\cap \B 1 0$ such that
\begin{enumerate}
 \item For all $i=1,2,\cdots$
 \begin{gather}
  \cS\subseteq \bigcup_{x\in \cC^{(i,r)}} \B {r}{x}  \cup  \bigcup_{x\in \cC^{(i,f)}} \B {r_x}{x}\cup  \bigcup_{x\in \cC^{(i,b)}} \B {r_x}{x}\, .
 \end{gather}
 \item For all $x\in \cC^{(i,r)}$, $r_x=r$. In other words, on these ``$r$-balls'' option \eqref{it_rec_1} of our lemma is verified,
 \item For all $x\in \cC^{(i,f)}$ and all $z\in \B{2r_x}{x}$ we have $\hat \theta(z,r_x/10)\leq E-\delta$. In other words, on these ``final balls'' option \eqref{it_rec_2} of our lemma is verified, 
 \item \label{it_rho--} for all $x\in \cC^{(i,b)}$, $r<r_x \leq \rho^i$. On these ``bad balls'', none of the two stopping options is verified, thus we need to refine our covering here.
 \item For some constant $C_F(m)$, we have the estimates
\begin{gather}\label{eq_cov5}
 \sum_{x\in \cC^{(i,r)}\cup \cC^{(i,f)} } r_x^k \leq C_F(m)\ton{\sum_{j=1}^i 2^{-j} }\, ,\quad \sum_{x\in \cC^{(i,b)} }r_x^k \leq 2^{-i}\, .
\end{gather}
 Thus the estimates on $r$ and final balls has uniform bounds, while our estimates on bad balls has exponentially decreasing bounds.
\end{enumerate}

\subsubsection{Re-covering of bad balls: First step in the induction}
For $i=1$, consider the covering \eqref{eq_covering_start} given by the previous lemma. We keep the balls $\cur{\B{r_x}{x}}_{x\in \cC_r}$ as they are, while for each $x\in \cC_+$ consider two coverings of $\B {\rho r_x}{F_x}$ and its complement
\begin{gather}
 \B {r_x}{x}\setminus \B{\rho r_x}{F_x} \subseteq \bigcup_{y\in \cC_x^{(1,f)}} \B{\rho r_x}{y}\, , \quad \B {r_x}{x}\cap \B{\rho r_x}{F_x} \subseteq \bigcup_{y\in \cC_x^{(1,b)}} \B{\rho r_x}{y}\, ,
\end{gather}
where $\B {\rho r_x/2}{y}$ are pairwise disjoint in both coverings. 

By definition of $F_x$, for all $y\in \cC_x^{(1,f)}$ the energy drop condition \eqref{eq_Edrop} is satisfied. Moreover we have the trivial estimates
\begin{gather}
 \sum_{y\in \cC_x^{(1,f)}} (\rho r_x)^k = \ton{\rho r_x}^k \#\cur{y\in \cC_x^{(1,f)}}\leq c(m) \rho^{k-m} r_x^k \equiv C_f(m,\rho) r_x^k\, .
\end{gather}
Since the energy drop is verified on these balls, we define $\cC^{(1,f)}$ to be the set of final balls at the step $i=1$ by
\begin{gather}
 \cC^{(1,f)} = \bigcup_{x\in \cC_+} \cC_x^{(1,f)}\, .
\end{gather}

For $y\in \cC_x^{(1,b)}$, the energy drop condition is not verified. However, since there exists a $k-1$ dimensional space $L_x$ such that
\begin{gather}
  F_x = \cur{y\in \cS \cap \B {2r_x}{x} \ \ s.t. \ \ \hat \theta(y,\rho r_x/10)\geq E-\delta}\subseteq \B {\rho r_x /5}{L_x}\, ,
\end{gather}
then we can estimate 
\begin{gather}\label{eq_Cc}
 \sum_{y\in \cC^{(1,b)}_x} (\rho r_x)^k =\rho^k r_x^k \#\cur{\cC^{(1,b)}_x}\leq c(m) \rho^{1-k} \rho^k r_x^k \equiv C_c(m) \rho r_x^k\, .
\end{gather}
On these balls, we can either have the stopping condition $\rho r_x=r$, or we need to refine the covering further. Thus we define
\begin{gather}
 \cC^{(1,b)}=\bigcup_{x\in \cC_+\, ,\ \ \rho r_x >r} \cC^{(1,b)}_x\, , \quad \cC^{(1,r)}=\cC_r \cup \bigcup_{x\in \cC_+\, ,  \ \ \rho r_x =r}  \cC^{(1,b)}_x\, .
\end{gather}
$\cC^{(1,b)}$ represents the set of ``bad balls'' where we need to refine our covering further.

By this and lemma \ref{lemma_cover}, in particular by the estimates in \eqref{eq_covering_start}, we obtain that
\begin{gather}\label{eq_bad1}
  \sum_{y\in \cC^{(1,b)} }r_y^k\leq C_c(m)\rho\sum_{x\in \cC_+} r_x^k \leq C_V(m) C_c(m) \rho\, .
\end{gather}
If we choose
\begin{gather}\label{eq_rho}
 0<\rho(m) \leq \min\cur {100^{-1},\frac {1} 2 C_V(m)^{-1} \cdot C_c(k)^{-1}}\, ,
\end{gather}
we can rephrase the above estimates as
\begin{gather}\label{eq_bad2}
 \sum_{y\in \cC^{(1,b)} }r_y^k \leq \frac 1 2\, .
\end{gather}
If we set 
\begin{gather}\label{eq_CF}
 C_F(m) = 2 C_V(m)\ton{C_f(m,\rho(m))+C_c(m)}\, ,
\end{gather}
the estimates on the final and $r$-balls are
\begin{gather}\label{eq_rf}
 \sum_{y\in \cC^{(1,r)}\cup \cC^{(1,f)} } r_y^k =\#\cur{\cC_r} r^k + \sum_{x\in \cC_+} r_x^k \ton{C_f(m,\rho)+C_c(m,\rho)}\leq C_V(m)\ton{C_f(m,\rho)+C_c(m)}= \frac 1 2 C_F(m)\, .
\end{gather}

Note that clearly for all $y\in \cC^{(1,b)}$, we have $r_y\leq \rho$.

\subsubsection{Re-covering of bad balls: Induction step}

Suppose that we have obtained our covering for $i$. It is clear that we need to improve our covering only on the balls $\cur{\B {r_x}{x}}_{x\in \cC^{(i,b)}}$. In order to do so, we consider each of these balls separately.

Since all the assumptions on lemma \ref{lemma_cover} are satisfied on each of the $\B{r_x}{x}$, we can apply again this lemma to each $\B{r_x}{x}$, and obtain that for all $x$ there exists a covering
\begin{gather}\label{eq_covering_start_i}
 \cS\cap \B {r_x}{x} \subseteq \bigcup_{y\in \hat \cC_{r,x}} \B {r}{y}\cup \bigcup_{y\in \hat \cC_{+,x}} \B {r_y} {y} \quad \text{ with } \ \ r_y\geq r \ \ \text{ and}\ \ \sum_{y\in \hat \cC_{r,x}\cup \hat \cC_{+,x}} r_y^k\leq C_V(m)r_x^k\, .
\end{gather}
Moreover, for each $y\in \hat \cC_{+,x}$, there exists a $k-1$ dimensional subspace $L_y$ such that
\begin{gather}
 F_y\equiv\cur{z\in \cS\cap \B {2r_y}{y} \ \ s.t. \ \ \theta(z,\rho r_y/10)>E-\delta }\subseteq \B {\rho r_y/5}{L_y}\cap \B {2r_y}{y}\, .
\end{gather}
By applying exactly the same procedure described in the first step of the induction to each of the balls $\cur{\B {r_y}{y}}_{y\in \hat \cC_{+,x}}$, we obtain the new desired covering. In particular, for each $y\in \hat \cC_{+,x}$ we can find a covering
\begin{gather}
 \B {r_y}{y}\setminus \B{\rho r_y}{F_y} \subseteq \bigcup_{z\in \hat \cC_y^{(i+1,f)}} \B{\rho r_y}{z}\, , \quad \B {r_y}{y}\cap \B{\rho r_y}{F_y} \subseteq \bigcup_{z\in \cC_y^{(i+1,b)}} \B{\rho r_y}{z}\, ,
\end{gather}
where for all $z\in \hat \cC_y^{(i+1,f)}$ and all $p\in \cS\cap \B {2\rho r_y}{z}$, we have $\hat \theta(p,\rho r_y)\leq E-\delta$, and we have the estimates
\begin{gather}
 \sum_{z\in \hat \cC_y^{(i+1,f)}} (\rho r_y)^k \leq C_f(m,\rho) r_y^k\, , \quad \sum_{z\in \cC^{(i+1,b)}_y} (\rho r_y)^k \leq C_c(m) \rho r_y^k\, .
\end{gather}

The new set $\cC^{(i+1,f)}$ is now defined as the previous set of ``final balls'' $\cC^{(i,f)}$ along with the new final balls $\hat \cC^{(i+1,f)}$ obtained with this covering, thus making
\begin{gather}
 \hat \cC^{(i+1,f)} = \bigcup_{x\in \cC^{(i,b)}} \bigcup_{y\in \hat \cC_{+,x}}\hat \cC_y^{(i+1,f)}\, , \quad \cC^{(i+1,f)} = \cC^{(i,f)} \cup \hat \cC^{(i+1,f)}\, .
\end{gather}
In a similar way for the $r$-balls, we obtain
\begin{gather}
  \quad \hat \cC^{(i+1,r)}=\bigcup_{x\in \cC^{(i,b)}} \ton{\hat \cC_{r,x} \cup \bigcup_{y\in \hat \cC_{+,x}\, ,  \ \ \rho r_y =r}  \cC^{(i+1,b)}_y}\, , \quad \cC^{(i+1,r)} = \cC^{(i,r)} \cup \hat \cC^{(i+1,r)} \, .
\end{gather}
However, evidently the new set of ``bad balls'' does not contain the bad balls at the previous scale, since those are the ones that were just re-covered. In particular
\begin{gather}
 \cC^{(i+1,b)}=\bigcup_{x\in \cC^{(i,b)}} \bigcup_{y\in \cC_{+,x},\ \rho r_y >r} \cC^{(i+1,b)}_y\, .
\end{gather}

The $k$-dimensional content estimate of our covering are obtained by iterating the estimates obtained in the first step. In detail, by arguing as in \eqref{eq_bad1} and \eqref{eq_bad2}, and by choosing $\rho$ according to \eqref{eq_rho}, we obtain
\begin{gather}
 \sum_{z\in \cC^{(i+1,b)} }r_z^k \leq \sum_{x\in \cC^{(i,b)}} \frac 1 2 r_x^k = 2^{-1-i}\, .
\end{gather}
As for final and $r$-balls, arguing as in \eqref{eq_rf} we can estimate the contribution given by the new $r$ and final balls by
\begin{gather}
 \sum_{z\in \hat \cC^{(i+1,r)} \cup \hat \cC^{(i+1,f)}} r_z^k \leq \ton{\frac 1 2 C_F(m)}\sum_{z\in \cC^{(i,b)} } r_z^k = 2^{-i-1} C_F(m)\, .
\end{gather}
This yields the desired result \eqref{eq_cov5}, and in turn concludes the proof of the lemma.

It is worth noticing that at the $i$-th step of the induction, the radius of the biggest ball in the covering is smaller than $\rho^i$. Thus eventually $\rho^i \leq r$ and this induction will stop in a finite number of steps.

\end{proof}

\subsubsection{Keeping track of the constants}

For the reader's convenience we record here how all the constants involved in the previous two lemmas were chosen.

First of all, note that $\epsilon>0$ is arbitrary, as well as $r>0$. However, it is of course important that all the constants here are \textit{independent of} $r$.

$C_R(m)$ is the constant coming from the Reifenberg theorem \ref{th_disc_reif}, and it depends only on $m$. $C_V(m)$ is fixed in \eqref{eq_CV}, and it is just a dimensional constant $c(m)$ (coming from a rough cover of $\B 1 0$ by balls of radius $1/8$) times $C_R(m)$. Thus $C_V(m)$ clearly depends only on $m$. 
$C_c(m)$ is fixed in \eqref{eq_Cc}, and is just another covering constant whose value depends only on $m$.

The parameter $\rho$, which was a free parameter in the first covering, is fixed once and for all in \eqref{eq_rho} as a constant depending only on $m$. For convenience, we can also pick a $\rho$ satisfying \eqref{eq_convenience}. Once this choice has been fixed, also the constant $C_F(m)$ introduced in \eqref{eq_CF} depends only on $m$.

The parameter $\eta>0$ is chosen according to \eqref{eq_eta_last}, \eqref{eq_eta0} and \eqref{eq_eta1}, as explained in \eqref{eq_eta}. At last, with this positive value of $\eta$ fixed, we choose $\delta$ in such a way that \eqref{eq_da}, \eqref{eq_db} and \eqref{eq_eta_last} are all satisfied.

\subsection{Proof of the main theorems}
Before proving our main theorems, we provide an argument that justifies the assumption $\cF<\delta$ which is present in all of our technical lemmas and covering arguments. The idea is that by condition \eqref{eq_f}, we can focus on small enough scales $r$ on which the value of $\cF$ has ``decayed'' by a factor $r^\gamma$.

\subsubsection{First covering by balls of small radius}\label{sec_delta}
 In all the estimates we need, an important assumption is that the constant $\cF$ in \eqref{eq_f} is sufficiently small, in other words our estimates apply if $u$ is an approximate harmonic map and the error $f$ is small enough. This assumption is not too restrictive because \eqref{eq_f} is better than scale invariant in nature. Indeed, if we restrict ourselves to small enough scales $r\leq r_0$, the rescaled maps $T^u_{x,r}:\B 3 0 \to N$ are approximate harmonic maps solving \eqref{eq_T1} and \eqref{eq_T2} with \eqref{eq_Tf}, and the error function $\tilde f$ satisfies
 \begin{gather}
  s^{4-m}\int_{\B y s} \abs{\tilde f}^2 \leq (\cF r_0^\gamma) s^\gamma
 \end{gather}
 for all $y\in \B 1 0$ and $s\leq 1$. Thus, if we choose $r_0(m,\cF,\gamma,\delta)$ in such a way that $\cF r_0^\gamma\leq \delta$, we can guarantee the smallness hypothesis $\cF<\delta$ on all smaller scales. 

 We can cover the original ball $\B 1 0$ with balls $\B {r_0/2}{x_i}$ such that $\B{r_0/4}{x_i}$ are disjoint, and then start over on each of this smaller balls. Evidently, we have
 \begin{gather}
  \sum_i r_0^k \leq C(m) r_0^{k-m}\leq C_0(m,\cF,\gamma,\delta)\, .
 \end{gather}
 Since we will pick the parameter $\delta$ from the covering Lemma \ref{lemma_coverII}, where we have $\delta=\delta(m,\KN,\Lambda,\gamma,\epsilon)$, and since we are only doing this rough covering once on the first big ball, the final estimates of Theorem \ref{th_mink} are not modified by this. Also the statements about the rectifiability (which is stable under countable unions, let alone finite unions) is not effected by this covering.
 
Now we are in a position to prove our main theorems. We start with the volume estimates of Theorem \ref{th_mink}.

\subsubsection{Proof of Theorem \ref{th_mink}}
This proof is basically a corollary of the covering Lemma \ref{lemma_coverII}. 

By the argument in Section \ref{sec_delta}, we can assume that $\cF<\delta$ throughout this proof. Moreover, in a similar spirit, instead of the estimates of \eqref{eq_mink}, we will prove the slightly less powerful estimate
\begin{gather}\label{eq_mink_lesspower}
\Vol\ton{\B r {\cS^k_{\epsilon,\delta r}(u) } \cap \B 1 0 }\leq C_\epsilon' r^{n-k}\, ,
\end{gather}
the difference being the $\delta r $ in $\cS^k_{\epsilon,\delta r}$. As above, since $\delta=\delta(m,\KN,\Lambda,\gamma,\epsilon)$, this does not affect the final estimate in \eqref{eq_mink}, if not by enlarging the constant $C_\epsilon'$ to $C_\epsilon$.

Consider the set $\cS=\cS^{k}_{\epsilon,\delta r}(u)\cap \B 1 0$. Note that by the monotonicity of $\hat \theta$ and the estimates in Lemma \ref{lemma_monotone}, we have the uniform bounds
\begin{gather}
\forall x\in \B 1 0\, , \ \ \forall r\in [0,1]\, , \quad \hat \theta(x,r)\leq \Lambda'=c(m)\Lambda + c(m,\gamma) \cF \, .
\end{gather}
Let $E=\sup_{x\in \cS} \hat \theta(x,1)  \leq \Lambda'$.

\subsubsection{Induction on energy upper bounds}

Using the covering Lemma \ref{lemma_coverII}, we will prove by induction on $i=0,1,\cdots,\lfloor \delta^{-1} E \rfloor +1 $ that there exist coverings of $\cS$ by balls $\cur{\B{r_x}{x}}_{x\in \cC^i}$ such that
\begin{gather}\label{eq_final_cov_1}
 \cS\subseteq \bigcup_{c\in \cC^i} \B {r_x}{x}\, , \quad \sum_{x\in \cC^i} r_x^k \leq (c(m)C_F(m))^i\, .
\end{gather}
Moreover, for all $i$ we have
\begin{gather}\label{eq_final_cov_2}
 r_x \leq r \quad \text{ or } \quad \forall y\in \cS \cap \B {2r_x}{x}\, , \ \ \hat \theta(y,r_x)\leq E- i \delta\, .
\end{gather}
It is clear that if we pick $i = \lfloor \delta^{-1} E \rfloor +1$, then the second condition cannot be true anywhere, which means that the first condition must be true, which will complete the construction of the covering. 

Now it is clear that the estimate \eqref{eq_mink_lesspower} is true with
\begin{gather}
 C_\epsilon' (m,\KN,\Lambda,\cF,\gamma,\epsilon)= 2^k(c(m)C_F(m))^{\lfloor \delta^{-1} E \rfloor +1}\, .
\end{gather}

So the only thing left to do is to prove the properties of this inductive covering in \eqref{eq_final_cov_1}. Note that this covering is trivial for $i=0$, since $\cS \subseteq \B 1 0$ does the trick at this stage.

By induction, suppose that \eqref{eq_final_cov_1} and \eqref{eq_final_cov_2} are true for $i$. Pick any $x\in \cC^i$, and consider $\B {r_x}{x}$. By the covering lemma \ref{lemma_coverII} (or better, by an $r_x$-rescaled version of this lemma), there exists a covering $\hat \cC_x$ of $\cS \cap \B {r_x}{x}$ such that
\begin{gather}
 \cS \cap \B {r_x}{x} \subseteq \bigcup_{y\in \hat \cC_x} \B{r_y}{y}\, , \quad r_y\leq \rho r_x\leq \rho^i \, , \quad \sum_{y\in \hat \cC_x} r_y^k \leq C_F(m) r_x^k\, .
\end{gather}
Moreover, for all $y\in \hat \cC_x$, we have
\begin{gather}
 \text{either} \ r_y=r \ \ \ \ \text{or} \ \ \  \ \forall z\in \cS \cap \B {2r_y}{y}\, , \ \  \hat \theta(z,\rho r_y/10)\leq E-i\delta - \delta = E-(i+1)\delta\, .
\end{gather}
By covering each $\B {r_y}{y}$ again by a minimal set of balls of radius $\rho(m) r_y \leq r_y/10$, we obtain a covering $\cC_x$ such that
\begin{gather}
 \cS \cap \B {r_x}{x} \subseteq \bigcup_{y\in \cC_x} \B{r_y}{y}\, , \quad r_y\leq \rho r_x\leq \rho^i \, , \quad \sum_{y\in \cC_x} r_y^k \leq c(m)C_F(m) r_x^k\, .
\end{gather}
Moreover, for all $y\in \cC_x$, we have
\begin{gather}
 \text{either} \ r_y\leq r \ \ \ \ \text{or} \ \ \  \ \forall z\in \cS \cap \B {2r_y}{y}\, , \ \  \hat \theta(z,\rho r_y)\leq E-i\delta - \delta = E-(i+1)\delta\, .
\end{gather}

By summing all the contributions coming from balls $\cur{\B{r_x}{x}}_{x\in \cC^{i}}$, we obtain
\begin{gather}
 \cC^{i+1} =\bigcup_{x\in \cC^i} \cC_x\, , \quad \sum_{y\in \cC^{i+1}} r_y^k =\sum_{x\in \cC^i } \ton{\sum_{y\in \cC_x} r_y^k}\leq (c(m)C_F(m))^{i+1}\, ,
\end{gather}
as desired.

\subsubsection{Proof of Theorem \ref{th_rect} and rectifiability of \texorpdfstring{$\cS^k_\epsilon$}{the quantitative strata}}
By countable additivity, the rectifiability of $\cS^k(u)$ is a corollary of the rectifiability of $\cS^k_\epsilon(u)$ for all $\epsilon>0$. 

By the volume estimates in \eqref{eq_mink}, we have $\lambda^k \ton{\cS^k_\epsilon \cap \B 1 0} \leq C_\epsilon$. By applying the same estimates on any ball $\B r x $ with $x\in \B 1 0$ and $r\leq 1$, we obtain that
\begin{gather}\label{eq_alreg}
 \lambda^k\ton{\cS^k_\epsilon \cap \B r x} \leq C_\epsilon r^k\, ,
\end{gather}
in other words, $\cS^k_\epsilon$ is upper-\al regular.

We will prove that for all measurable subsets $\cS\subseteq \cS^k_\epsilon\cap \B 1 0$, there exists a $k$-measurable subset $E\subset \cS$ with $\lambda^k(E)\leq 7^{-1}\lambda^k(\cS)$ such that $\cS\setminus E$ is $k$-rectifiable. Since $\cS$ is an arbitrary measurable subset, this is enough to prove rectifiability by a standard density argument.

Consider any $\cS\subseteq \cS^k_\epsilon \cap \B 1 0$. We can assume wlog that $\lambda^k(\cS)>0$, otherwise there is nothing to prove. Consider the function $f(x,r)=\hat \theta(x,r)-\hat \theta(x,0)$ on $\B 1 0$. This function is monotone nondecreasing in $r$, uniformly bounded for all $x\in \B 1 0$ and $r\leq 1$, and pointwise converging to $0$ as $r\to 0$.

Thus, by dominated convergence, for all $\delta>0$, there exists a radius $\bar r >0$ such that
\begin{gather}
 \fint_{\cS} f(x,10 \bar r)d\lambda^k(x) \leq \delta^2\, .
\end{gather}
Let $E\subset \cS$ be a measurable subset with $\lambda^k(E)\leq \delta \lambda^k(\cS)$ and such that $f(x,10 \bar r)\leq \delta$ for all $x\in F\equiv \cS\setminus E$.

Now cover $F$ by a finite number of balls $\B {\bar r}{x_i}$ centered on $F$. We want to show that, if $\delta$ is chosen small enough, then on each of these balls we can apply Theorem \ref{th_rect_reif} to $F\cap \B {\bar r}{x_i}$ for all $i$, and thus proving that $F$ is $k$-rectifiable as desired.

\paragraph{Reifenberg estimates}
The estimates here are basically equivalent to the estimates carried out in Section \ref{sec_volest}. Actually, since we already know that \eqref{eq_alreg} holds, we do not even need the upper induction part of that argument. For this reason, we will only sketch the main passages in the estimates.

Fix any $i$, and consider the set $F\cap \B {\bar r}{x_i}$. For convenience, we rescale the ball $\B {\bar r }{x_i}$ to $\B 1 0$. With an abuse of notation, we will keep denoting by $u$, $\hat \theta$, $\cS^k_\epsilon$ and $F$ also the rescaled objects.

By definition of $F\subset \cS^k_{\epsilon}$, we have that $\hat \theta(x,10)-\hat\theta(x,0)\leq \delta$ for all $x\in F$. By an estimate analogous to \eqref{eq_vol1}, we have for all $x\in F$ and $s\leq 1$
\begin{gather}
 \beta_{2,\lambda^k|_F} (x,s)^2 \leq C_1 s^{-k}\int_{\B {s}{x}} W_s (y)\,d\lambda^k|_F (y)
\end{gather}
By integrating, and by \eqref{eq_alreg}, we obtain for all $x\in \B 1 0$ and $s\leq r \leq 1$:
\begin{gather}
 \int_{\B r x} \beta_{2,\lambda^k|_F} (z,s)^2\, d\lambda^k|_F(z) \leq C_1 s^{-k}\int_{\B r y} \qua{\int_{\B {s}{z}} W_s (t)\, d\lambda^k|_F(t)} d\lambda^k|_F(z) \leq C_1C_\epsilon \int_{\B {r+s} y } W_s (z)d\lambda^k|_F (z)\, .
\end{gather}
Integrating again in $s$, we finally get for all $x\in \B 1 0$ and $r\leq 1$:
\begin{gather}
 \int_{\B r x} \qua{\int_0^s  \beta_{2,\lambda^k|_F} (z,s)^2\,\frac {ds}{s}} d\lambda^k|_F(z)\leq C_1C_\epsilon \int_{\B {2r} x} [\hat \theta(x,8r)-\hat \theta(x,0) ]d\lambda^k|_F(z)\leq c(m) C_1C_\epsilon^2 \delta r^k\, .
\end{gather}
By choosing
\begin{gather}
 \delta \leq \frac{\delta_R^2}{c(m) C_1 C_\epsilon^2}\, ,
\end{gather}
we can apply Theorem \ref{th_rect_reif} to the set $F\cap \B 1 0$, thus proving that it is $k$-rectifiable. This concludes the proof.

\bibliographystyle{aomalpha}
\bibliography{../qstrat}
\vspace{.5 cm}

\end{document}